\documentclass[journal]{IEEEtran}

\usepackage{setspace}
\usepackage{lettrine}
\usepackage{graphics,
           psfrag,
           epsfig,
           amsthm,
           cite,
           amssymb,
           url,
           dsfont,
           subfigure,
           algorithm,
           algorithmic,
           balance,
           enumerate,
           color,
           setspace
}
\usepackage{amsmath}%[centertags][tbtags][sumlimits][nosumlimits][intlimits][namelimits][nonamelimits]

\usepackage{epstopdf}

\newtheorem{definition}{Definition}

\newtheorem{lemma}{Lemma}

\newtheorem{theorem}{Theorem}
\newtheorem{corollary}{Corollary}
\newtheorem{proposition}{Proposition}

\newtheorem{preposition}{Preposition}
\newtheorem{remark}{Remark}

\newcommand{\preref}[1]{Preposition~\ref{#1}}
\newcommand{\eref}[1]{(\ref{#1})}
\newcommand{\sref}[1]{Section~\ref{#1}}
\newcommand{\appref}[1]{Appendix~\ref{#1}}
\newcommand{\fref}[1]{Figure~\ref{#1}}
\newcommand{\dref}[1]{Definition~\ref{#1}}

\newcommand{\cref}[1]{Constraint~\ref{#1}}
\newcommand{\thref}[1]{Theorem~\ref{#1}}
\newcommand{\corref}[1]{Corollary~\ref{#1}}
\newcommand{\lref}[1]{Lemma~\ref{#1}}
\newcommand{\rref}[1]{Remark~\ref{#1}}
\newcommand{\tref}[1]{Table~\ref{#1}}
\newcommand{\algref}[1]{Algorithm~\ref{#1}}

\newcommand{\ignore}[1]{}

\IEEEoverridecommandlockouts
\usepackage{slashbox}

\begin{document}

\title{\vspace{-.5cm}Manifold Optimization Over the Set of Doubly Stochastic Matrices: A Second-Order Geometry}

\author{
   \IEEEauthorblockN{Ahmed Douik, \textit{Student Member, IEEE} and Babak Hassibi, \textit{Member, IEEE}\vspace{-.8cm}}

\thanks {
Ahmed Douik and Babak Hassibi are with the Department of Electrical Engineering, California Institute of Technology, Pasadena, CA 91125 USA (e-mail: \{ahmed.douik,hassibi\}@caltech.edu).
}
}

\maketitle

\begin{abstract}
Convex optimization is a well-established research area with applications in almost all fields. Over the decades, multiple approaches have been proposed to solve convex programs. The development of interior-point methods allowed solving a more general set of convex programs known as semi-definite programs and second-order cone programs. However, it has been established that these methods are excessively slow for high dimensions, i.e., they suffer from the curse of dimensionality. On the other hand, optimization algorithms on manifold have shown great ability in finding solutions to nonconvex problems in reasonable time. This paper is interested in solving a subset of convex optimization using a different approach. The main idea behind Riemannian optimization is to view the constrained optimization problem as an unconstrained one over a restricted search space. The paper introduces three manifolds to solve convex programs under particular box constraints. The manifolds, called the doubly stochastic, symmetric and the definite multinomial manifolds, generalize the simplex also known as the multinomial manifold. The proposed manifolds and algorithms are well-adapted to solving convex programs in which the variable of interest is a multidimensional probability distribution function. Theoretical analysis and simulation results testify the efficiency of the proposed method over state of the art methods. In particular, they reveal that the proposed framework outperforms conventional generic and specialized solvers, especially in high dimensions.
\end{abstract}


\section{Introduction} \label{sec:int}

Numerical optimization is the foundation of various engineering and computational sciences. Consider a mapping $f$ from a subset $\mathds{D}$ of $\mathds{R}^n$ to $\mathds{R}$. The goal of the optimization algorithms is to find an extreme point $x^* \in \mathds{D}$ such that $f(x^*) \leq f(y)$ for all point $y \in \mathcal{N}_{x^*}$ in the neighborhood of $x^*$. Unconstrained Euclidean\footnote{The traditional optimization schemes are identified with the word Euclidean in contrast with the Riemannian algorithm in the rest of the paper.} optimization refers to the setup in which the domain of the objective function is the whole space, i.e., $\mathds{D}=\mathds{R}^n$. On the other hand, constrained Euclidean optimization denotes optimization problem in which the search set is constrained, i.e., $\mathds{D}\subsetneq \mathds{R}^n$. 

Convex optimization is a special case of constrained optimization problems in which both the objective function and the search set are convex. Historically initiated with the study of least-squares and linear programming problems, convex optimization plays a crucial role in optimization algorithm thanks to the desirable convergence property it exhibits. The development of interior-point methods allowed solving a more general set of convex programs known as semi-definite programs and second-order cone programs. A summary of convex optimization methods and performance analysis can be found in the seminal book \cite{993483}.

Another important property of convex optimization is that the interior of the search space can be identified with a manifold that is embedded in a higher-dimensional Euclidean space. Using advanced tools to solve the constrained optimization, e.g., \cite{scheithauerjorge}, requires solving on the high dimension space which can be excessively slow. Riemannian optimization takes advantage of the fact that the manifold is of lower dimension and exploits its underlying geometric structure. The optimization problem is reformulated from a constrained Euclidean optimization into an unconstrained optimization over a restricted search space, a.k.a., a Riemannian manifold. 

Thanks to the aforementioned low-dimension feature, optimization over Riemannian manifolds is expected to perform more efficiently \cite{11082885}. Therefore, a large body of literature dedicated to adapting traditional Euclidean optimization methods and their convergence properties to Riemannian manifolds. 

This paper introduces a framework for solving optimization problems in which the optimization variable is a doubly stochastic matrix. Such framework is particularly interesting for clustering applications. In such problems ,e.g., \cite{7541267,Yang:2012:CLD:3042573.3042666,Wang:2016,zass2006doubly}, one wishes to recover the structure of a graph given a similarity matrix. The recovery is performed by minimizing a predefined cost function under the constraint that the optimization variable is a doubly stochastic matrix. This work provides a unified framework to carry such optimization.

\subsection{State of the Art}

Optimization algorithms on Riemannian manifolds appeared in the optimization literature as early as the 1970's with the work of Luenberger \cite{5485452620} wherein the standard Newton's optimization method has been adapted to problems on manifolds. A decade later, Gabay \cite{Gabay1982} introduces the steepest descent and the quasi-Newton algorithm on embedded submanifolds of $\mathds{R}^n$. The work investigates the global and local convergence properties of both the steepest descent and the Newton's method. The analysis of the steepest descent and the Newton algorithm is extended in \cite{165579,udriste1994convex} to Riemannian manifolds. By using exact line search, the authors concluded the convergence of their proposed algorithms. The assumption is relaxed in \cite{Yang2007} wherein the author provides convergence rate and guarantees for the steepest descent and Newton's method for Armijo step-size control.

The above-mentioned works substitute the concept of the line search in Euclidian algorithms by searching along a geodesic which generalizes the idea of a straight line. While the method is natural and intuitive, it might not be practical. Indeed, finding the expression of the geodesic requires computing the exponential map which may be as complicated as solving the original optimization problem \cite{AbsMahSep2008}. To overcome the problem, the authors in \cite{1399106} suggest approximating the exponential map up to a given order, called a retraction, and show quadratic convergence for Newton's method under such setup. The work initiated more sophisticated optimization algorithm such as the trust region methods \cite{BakAbsGal2008,baker2008riemannian,absil2007trust,yootrust,11082885}. Analysis of the convergence of first and second order methods on Riemannian manifolds, e.g., gradient and conjugate gradient descent, Newton's method, and trust region methods, using general retractions are summarized in \cite{AbsMahSep2008}. 

Thanks to the theoretical convergence guarantees mentioned above, the optimization algorithms on Riemannian manifolds are gradually gaining momentum in the optimization field \cite{11082885}. Several successful algorithms have been proposed to solve non-convex problems, e.g., the low-rank matrix completion \cite{20111254,445872013,99874582015}, online learning \cite{2188399}, clustering \cite{14210041033,104692046925} and tensor decomposition \cite{7182334}. It is worth mentioning that these works modify the optimization algorithm by using a general connection instead of the genuine parallel vector transport to move from a tangent space to the other while computing the (approximate) Hessian. Such approach conserves the global convergence of the quasi-Newton scheme but no longer ensures their superlinear convergence behavior \cite{6515551}.

Despite the advantages cited above, the use of optimization algorithms on manifolds is relatively limited. This is mainly due to the lack of a systematic mechanism to turn a constrained optimization problem into an optimization over a manifold provided that the search space forms a manifold, e.g., convex optimization. Such reformulation, usually requiring some level of understanding of differential geometry and Riemannian manifolds, is prohibitively complex for regular use. This paper addresses the problem by introducing new manifolds that allow solving a non-negligible class of optimization problem in which the variable of interest can be identified with a multidimensional probability distribution function.

\subsection{Contributions}

In \cite{7182334}, in a context of tensor decomposition, the authors propose a framework to optimize functions in which the variable are stochastic matrices. This paper proposes extending the results to a more general class of manifolds by proposing a framework for solving a subset of convex programs including those in which the optimization variable represents a doubly stochastic and possibly symmetric and/or definite multidimensional probability distribution function. To this end, the paper introduces three manifolds which generalize the multinomial manifold. While the multinomial manifold allows representing only stochastic matrices, the proposed ones characterize doubly stochastic, symmetric and definite arrays, respectively. Therefore, the proposed framework allows solving a subset of convex programs. To the best of the author's knowledge, the proposed manifolds have not been introduced or studied in the literature.

The first part of the manuscript introduces all relevant concepts of the Riemannian geometry and provides insights on the optimization algorithms on such manifolds. In an effort to make the content of this document accessible to a larger audience, it does not assume any prerequisite on differential geometry. As a result, the definitions, concepts, and results in this paper are tailored to the manifold of interest and may not be applicable for abstract manifolds. 

The paper investigates the first and second order Riemannian geometry of the proposed manifolds endowed with the Fisher information metric which guarantees that the manifolds have a differentiable structure. For each manifold, the tangent space, Riemannian gradient, Hessian, and retraction are derived. With the aforementioned expressions, the manuscript formulates first and a second order optimization algorithms and characterizes their complexity. Simulation results are provided to further illustrate the efficiency of the proposed method against state of the art algorithms.

The rest of the manuscript is organized as follows: \sref{sec:opt} introduces the optimization algorithms on manifolds and lists the problems of interest in this paper. In \sref{sec:the}, the doubly stochastic manifold is introduced and its first and second order geometry derived. \sref{sec:the2} iterate a similar study to a particular case of doubly stochastic matrices known as the symmetric manifold. The study is extended to the definite symmetric manifold in \sref{sec:ext}. \sref{sec:opt2} suggests first and second order algorithms and analyze their complexity. Finally, before concluding in \sref{sec:con}, the simulation results are plotted and discussed in \sref{sec:sim}.

\section{Optimization on Riemannian Manifolds} \label{sec:opt}

This section introduces the numerical optimization methods on smooth matrix manifolds. The first part introduces the Riemannian manifold notations and operations. The second part extends the first and second order Euclidean optimization algorithm to the Riemannian manifolds and introduces the necessary machinery. Finally, the problems of interest in this paper are provided and the different manifolds identified. 

\subsection{Manifold Notation and Operations}

The study of optimization algorithms on smooth manifolds engaged a significant attention in the previous years. However, such studies require some level of knowledge of differential geometry. In this paper, only smooth embedded matrix manifolds are considered. Hence, the definitions and theorems may not apply to abstract manifolds. In addition, the authors opted for a coordinate free analysis omitting the chart and the differentiable structure of the manifold. For an introduction to differential geometry, abstract manifold, and Riemannian manifolds, we refer the readers to the following references \cite{543483584,54514584,9986155}, respectively.

An embedded matrix manifold $\mathcal{M}$ is a smooth subset of a vector space $\mathcal{E}$ included in the set of matrices $\mathds{R}^{n \times m}$. The set $\mathcal{E}$ is called the ambient or the embedding space. By smooth subset, we mean that the $\mathcal{M}$ can be mapped by a bijective function, i.e., a chart, to an open subset of $\mathds{R}^{d}$ where $d$ is called the dimension of the manifold. The dimension $d$ can be thought of as the \emph{degree of freedom} of the manifold. In particular, a vector space $\mathcal{E}$ is a manifold.

In the same line of though of approximating a function locally by its derivative, a manifold $\mathcal{M}$ of dimension $d$ can be approximated locally at a point ${\mathbf{X}}$ by a $d$-dimensional vector space $\mathcal{T}_{\mathbf{X}}\mathcal{M}$ generated by taking derivatives of all smooth curves going through ${\mathbf{X}}$. Formally, let $\gamma(t): \mathcal{I} \subset \mathds{R} \longrightarrow \mathcal{M}$ be a curve on $\mathcal{M}$ with $\gamma(0)={\mathbf{X}}$. Define the derivative of $\gamma(t)$ at zero as follows:
\begin{align}
\gamma^\prime(0) = \lim_{t \rightarrow 0} \cfrac{\gamma(t)-\gamma(0)}{t}.
\end{align}
The space generated by all $\gamma^\prime(0)$ represents a vector space $\mathcal{T}_{\mathbf{X}}\mathcal{M}$ called the tangent space of $\mathcal{M}$ at ${\mathbf{X}}$. \fref{fig:1} shows an example of a two-dimension tangent space generated by a couple of curves. The tangent space plays a primordial role in the optimization algorithms over manifold in the same way as the derivative of a function plays an important role in Euclidean optimization. The union of all tangent spaces $\mathcal{T}\mathcal{M}$ is referred to as the tangent bundle of $\mathcal{M}$, i.e.,:
\begin{align}
\mathcal{T}\mathcal{M} = \bigcup_{{\mathbf{X}} \in \mathcal{M}} \mathcal{T}_{\mathbf{X}}\mathcal{M}.
\end{align}

\begin{figure}[t]
\centering
\includegraphics[width=1\linewidth]{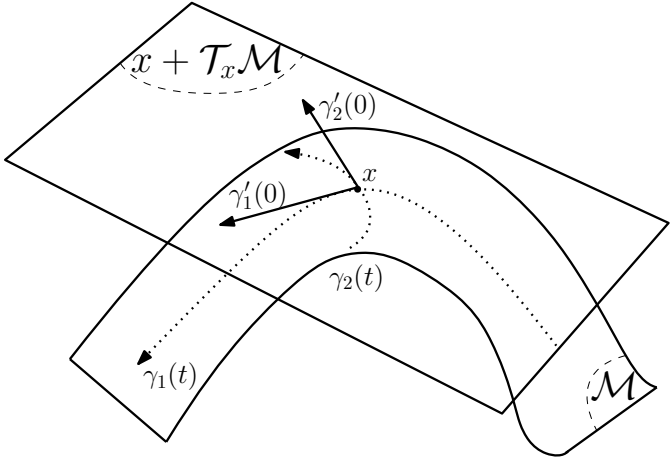}\\
\caption{Tangent space of a 2-dimensional manifold embedded in $\mathds{R}^3$. The tangent space $\mathcal{T}_{\mathbf{X}}\mathcal{M}$ is computed by taking derivatives of the curves going through ${\mathbf{X}}$ at the origin.} \label{fig:1}
\end{figure}

As shown previously, the notion of tangent space generalizes the notion of directional derivative. However, to optimize functions, one needs the notion of directions and lengths which can be achieved by endowing each tangent space $\mathcal{T}_{\mathbf{X}}\mathcal{M}$ by a bilinear, symmetric positive form $\langle.,.\rangle_{\mathbf{X}}$, i.e., an inner product. Let $g:\mathcal{T}\mathcal{M} \times \mathcal{T}\mathcal{M} \longrightarrow \mathds{R}$ be a smoothly varying bilinear form such that its restriction on each tangent space is the previously defined inner product. In other words:
\begin{align}
g(\xi_{\mathbf{X}},\eta_{\mathbf{X}}) =  \langle\xi_{\mathbf{X}},\eta_{\mathbf{X}}\rangle_{\mathbf{X}}, \ \forall \ \xi_{\mathbf{X}},\eta_{\mathbf{X}} \in \mathcal{T}_{\mathbf{X}}\mathcal{M}
\end{align}

Such metric, known as the Riemannian metric, turns the manifold into a Riemannian manifold. Any manifold (in this paper) admits at least a Riemannian metric. Lengths of tangent vectors are naturally induced from the inner product. The norm on the tangent space $\mathcal{T}_{\mathbf{X}}\mathcal{M}$ is denoted by $||.||_{\mathbf{X}}$ and defined by:
\begin{align}
||\xi_{\mathbf{X}}||_{\mathbf{X}} = \sqrt{\langle\xi_{\mathbf{X}},\xi_{\mathbf{X}}\rangle_{\mathbf{X}}} , \ \forall \ \xi_{\mathbf{X}} \in \mathcal{T}_{\mathbf{X}}\mathcal{M}
\end{align}

Both the ambient space and the tangent space being vector spaces, one can define the orthogonal projection $\Pi_{\mathbf{X}}: \mathcal{E} \longrightarrow \mathcal{T}_{\mathbf{X}}\mathcal{M}$ verifying $\Pi_{\mathbf{X}} \circ \Pi_{\mathbf{X}} = \Pi_{\mathbf{X}}$. The projection is said to be orthogonal with respect to the restriction of the Riemannian metric to the tangent space, i.e., $\Pi_{\mathbf{X}}$ is orthogonal in the $\langle.,.\rangle_{\mathbf{X}}$ sens.

\subsection{First and Second Order Algorithms}

The general idea behind unconstrained Euclidean numerical optimization methods is to start with an initial point ${\mathbf{X}}^0$ and to iteratively update it according to certain predefined rules in order to obtain a sequence $\{{\mathbf{X}}^t\}$ which converges to a local minimizes of the objective function. A typical update strategy is the following:
\begin{align}
{\mathbf{X}}^{t+1} = {\mathbf{X}}^t + \alpha^t p^t,
\end{align}
where $\alpha^t$ is the step size and $p^t$ the search direction. Let $\text{Grad }f({\mathbf{X}})$ be the Euclidean gradient\footnote{The expression of the Euclidean gradient (denoted by Grad) is explicitly given to show the analogy with the Riemannian gradient (denoted by grad). The nabla symbol $\nabla$ is not used in the context of gradient as it is reserved for the Riemannian connection. Similar notations are used for the Hessian.} of the objective function defined as the unique vector satisfying:
\begin{align}
\langle \text{Grad }f({\mathbf{X}}), \xi \rangle = \text{D} f({\mathbf{X}}) [\xi],\ \forall \ \xi \in \mathcal{E},
\end{align}
where $\langle.,.\rangle$ is the inner product on the vector space $\mathcal{E}$ and $\text{D} f({\mathbf{X}}) [\xi]$ is the directional derivative of $f$ given by:
\begin{align}
\text{D} f({\mathbf{X}}) [\xi] = \lim_{t \rightarrow 0} \cfrac{f({\mathbf{X}}+t\xi)-f({\mathbf{X}})}{t}
\end{align}

In order to obtain a descent direction, i.e., $f({\mathbf{X}}^{t+1}) < f({\mathbf{X}}^{t})$ for a small enough step size $\alpha^t$, the search direction $p^t$ is chosen in the half space spanned by $-\text{Grad }f({\mathbf{X}})$. In other words, the following expression holds:
\begin{align}
\langle \text{Grad }f({\mathbf{X}}^t), p^t \rangle < 0.
\label{eq:1}
\end{align}
In particular, the choices of the search direction satisfying
\begin{align}
p^t = -\cfrac{\text{Grad }f({\mathbf{X}}^t)}{||\text{Grad }f({\mathbf{X}}^t)||} \label{eq:4} \\
\text{Hess }f({\mathbf{X}}^t)[p^t] = \text{Grad }f({\mathbf{X}}) \label{eq:5}
\end{align}
yield the celebrated steepest descent \eref{eq:4} and the Newton's method \eref{eq:5}, wherein $\text{Hess }f({\mathbf{X}})[\xi]$ is the Euclidean Hessian\footnote{The Euclidean Hessian is seen as an operator to show the connection with the Riemanian Hessian. One can show that the proposed definition matches the ``usual" second order derivative matrix for $\xi=\mathbf{I}$.} of $f$ at ${\mathbf{X}}$ defined as an operator from $\mathcal{E}$ to $\mathcal{E}$ satisfying:
\begin{enumerate}
\item $\langle \text{Hess }f({\mathbf{X}})[\xi],\xi\rangle = \text{D}^2f({\mathbf{X}})[\xi,\xi]= \text{D}(\text{D}f({\mathbf{X}})[\xi])[\xi]$,
\item $\langle \text{Hess }f({\mathbf{X}})[\xi],\eta \rangle = \langle \xi, \text{Hess }f({\mathbf{X}})[\eta]\rangle ,\ \forall \ \xi,\eta \in \mathcal{E}$.
\end{enumerate}

After choosing the search direction, the step size $\alpha^t$ is chosen so as to satisfy the Wolfe conditions for some constant $c_1 \in (0,1)$ and $c_2 \in (c_1,1)$, i.e.,
\begin{enumerate}
\item The Armijo condition: 
\begin{align}
f({\mathbf{X}}^t+\alpha^t p^t) - f({\mathbf{X}}^{t}) \leq c_1 \alpha^t \langle\text{Grad }f({\mathbf{X}}^t),p^t \rangle
\label{eq:2}
\end{align}
\item The curvature condition: 
\begin{align}
\langle\text{Grad }f({\mathbf{X}}^t+\alpha^t p^t),p^t \rangle \geq c_2.
\label{eq:3}
\end{align}
\end{enumerate}

\begin{algorithm}[t!]
\begin{algorithmic}[1]
\REQUIRE Manifold $\mathcal{M}$, function $f$, and retraction $R$.
\STATE Initialize ${\mathbf{X}} \in \mathcal{M}$.
\WHILE {$||\text{grad }f({\mathbf{X}})||_{{\mathbf{X}}} \geq \epsilon$}
\STATE Choose search direction $\xi_{\mathbf{X}} \in \mathcal{T}_{\mathbf{X}}\mathcal{M}$ such that:
\begin{align}
\langle \text{grad }f({\mathbf{X}}), \xi_{\mathbf{X}} \rangle_{\mathbf{X}} <0.
\end{align}
\STATE Compute Armijo step size $\alpha$.
\STATE Retract ${\mathbf{X}} = R_{\mathbf{X}}(\alpha \xi_{\mathbf{X}})$.
\ENDWHILE
\STATE Output ${\mathbf{X}}$.
\end{algorithmic}
\caption{Line-Search Method on Riemannian Manifold}
\label{alg1}
\end{algorithm}

The Riemannian version of the steepest descent, called the line-search algorithm, follows a similar logic as the Euclidean one. The search direction is obtained with respect to the Riemannian gradient which is defined in a similar manner as the Euclidean one with the exception that it uses the Riemannian geometry, i.e.,:
\begin{definition}
The Riemannian gradient of $f$ at ${\mathbf{X}}$ denoted by $\text{grad }f({\mathbf{X}})$ of a manifold $\mathcal{M}$, is defined as the unique vector in $\mathcal{T}_{\mathbf{X}}\mathcal{M}$ that satisfies:
\begin{align}
\langle \text{grad }f({\mathbf{X}}), \xi_{\mathbf{X}} \rangle_{\mathbf{X}} = \text{D} f({\mathbf{X}}) [\xi_{\mathbf{X}}],\ \forall \ \xi_{\mathbf{X}} \in \mathcal{T}_{\mathbf{X}}\mathcal{M}.
\end{align}
\end{definition}

After choosing the search direction as mandated by \eref{eq:1}, the step size is selected according to Wolfe's conditions \emph{similar} to the one in \eref{eq:2} and \eref{eq:3}. A more general definition of a descent direction, known as gradient related sequence, and the Riemannian Armijo step expression can be found in \cite{AbsMahSep2008}. 

While the update step ${\mathbf{X}}^{t+1} = {\mathbf{X}}^t + \alpha^t p^t$ is trivial in the Euclidean optimization thanks to its vector space structure, it might result on a point ${\mathbf{X}}^{t+1}$ outside of the manifold. Moving on a given direction of a tangent space while staying on the manifold is realized by the concept of retraction. The ideal retraction is the exponential map $\text{Exp}_{\mathbf{X}}$ as it maps point a tangent vector $\xi_{\mathbf{X}} \in \mathcal{T}_{\mathbf{X}}\mathcal{M}$ to a point along the geodesic curve (straight line on the manifold) that goes through ${\mathbf{X}}$ in the direction of $\xi_{\mathbf{X}}$. However, computing the geodesic curves is challenging and may be more difficult that the original optimization problem. Luckily, one can use a first-order retraction (called simply retraction in this paper) without compromising the convergence property of the algorithms. A first-order retraction is defined as follows:
\begin{definition}
A retraction on a manifold $\mathcal{M}$ is a smooth mapping $R$ from the tangent bundle $\mathcal{T}\mathcal{M}$ onto $\mathcal{M}$. For all ${\mathbf{X}} \in \mathcal{M}$, the restriction of $R$ to $\mathcal{T}_{\mathbf{X}}\mathcal{M}$, called $R_{\mathbf{X}}$ satisfy the following properties:
\begin{itemize}
\item Centering: $R_{\mathbf{X}}(0)={\mathbf{X}}$.
\item Local rigidity: The curve $\gamma_{\xi_{\mathbf{X}}}(\tau) = R_{\mathbf{X}}(\tau \xi_{\mathbf{X}})$ satisfy $\cfrac{d\gamma_{\xi_{\mathbf{X}}}(\tau)}{d\tau}\Big|_{\tau=0} = \xi_{\mathbf{X}}, \ \forall \ \xi_{\mathbf{X}} \in \mathcal{T}_{\mathbf{X}}\mathcal{M}$.
\end{itemize}
\end{definition} 

\begin{figure}[t]
\centering
\includegraphics[width=0.8\linewidth]{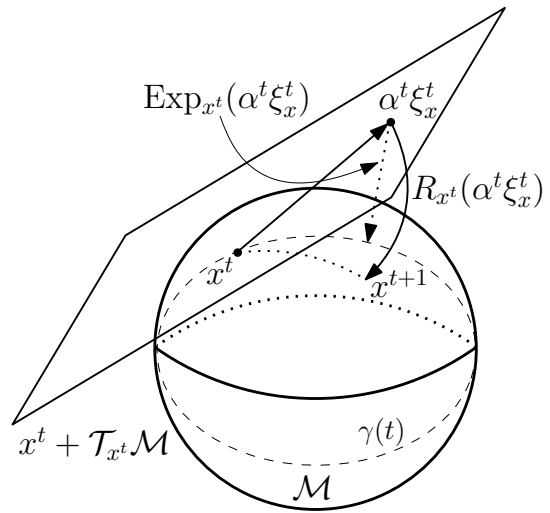}\\
\caption{The update step for the two-dimensional sphere embedded in $\mathds{R}^3$. The update direction $\xi_{\mathbf{X}}^t$ and step length $\alpha^t$ are computed in the tangent space $\mathcal{T}_{{\mathbf{X}}^t}\mathcal{M}$. The point ${\mathbf{X}}^t+\alpha^t\xi_{\mathbf{X}}^t$ lies outside the manifold and needs to be retracted to obtain the update ${\mathbf{X}}^{t+1}$. The update is not located on the geodesic $\gamma(t)$ due to the use of a retraction instead of the exponential map.} \label{fig:2}
\end{figure}

For some predefined Armijo step size, the procedure above is guaranteed to converge for all retractions \cite{AbsMahSep2008}. The generalization of the steepest descent to the Riemannian manifold is obtained by finding the search direction that satisfies similar equation as in the Euclidean scenario \eref{eq:4} using the Riemannian gradient. The update is then retracted to the manifold. The steps of the line-search method can be summarized in \algref{alg1} and an illustration of an iteration of the algorithm is given in \fref{fig:2}.

Generalizing the Newton's method to the Riemannian setting requires computing the Riemannian Hessian operator which requires taking a directional derivative of a vector field. As the vector field belong to different tangent spaces, one needs the notion of connection $\nabla$ that generalizes the notion of directional derivative of a vector field. The notion of connection is intimately related to the notion of vector transport which allows moving from a tangent space to the other as shown in \fref{fig:3}. The definition of a connection is given below:
\begin{definition}
An affine connection $\nabla$ is a mapping from $\mathcal{T}\mathcal{M} \times \mathcal{T}\mathcal{M}$ to $\mathcal{T}\mathcal{M}$ that associate to each $(\eta,\xi)$ the tangent vector $\nabla_\eta \xi$ satisfying for all smooth $f,g: \mathcal{M} \longrightarrow \mathds{R}$, $a,b \in \mathds{R}$:
\begin{itemize}
\item $\nabla_{f(\eta)+g(\chi)}\xi =  f(\nabla_\eta \xi)+ g(\nabla_\chi \xi)$
\item $\nabla_{\eta}(a\xi+b\varphi) = a\nabla_{\eta}\xi+b\nabla_{\eta}\varphi$
\item $\nabla_{\eta}(f(\xi)) = \xi(f) \eta + f(\nabla_\eta \xi)$,
\end{itemize}
wherein the vector field $\xi$ acts on the function $f$ by derivation, i.e., $\xi(f)=\text{D}(f)[\xi]$ also noted as $\xi f$ in the literature.
\end{definition}

\begin{figure}[t]
\centering
\includegraphics[width=0.7\linewidth]{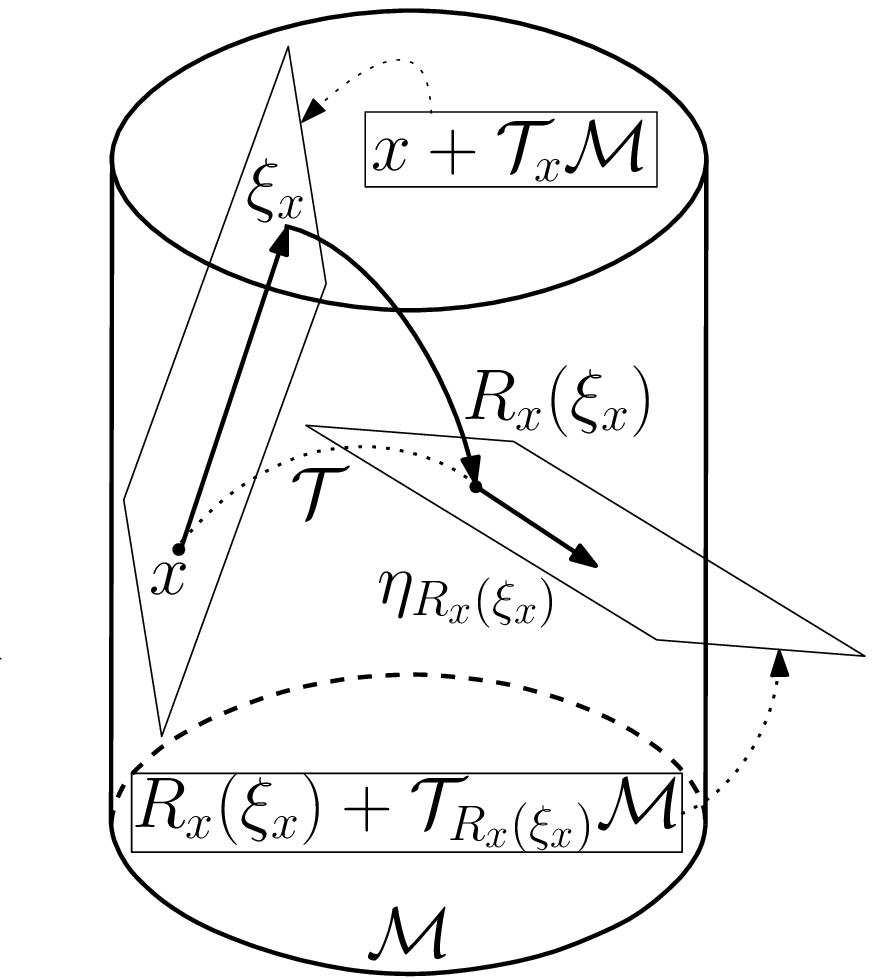}\\
\caption{An illustration of a vector transport $\mathcal{T}$ on a two-dimensional manifold embedded in $\mathds{R}^3$ that connects the tangent space of ${\mathbf{X}}$ with tangent vector $\xi_{\mathbf{X}}$ with the one of its retraction $R_{\mathbf{X}}(\xi_{\mathbf{X}})$. A connection $\nabla$ can be obtained from the speed at the origin of the inverse of the vector transport $\mathcal{T}^{-1}$.} \label{fig:3}
\end{figure}

On a Riemannian manifold, the Levi-Civita is the canonical choice as it preserve the Riemannian metric. The connection is computed as:
\begin{definition}
The Levi-Civita connection is the unique affine connection on $\mathcal{M}$ with the Reimannian metric $\langle.,.\rangle$ that satisfy for all $\eta ,\xi ,\chi \in \mathcal{T} \mathcal{M}$:
\begin{enumerate}
\item $\nabla_\eta \xi - \nabla_\xi \eta = [\eta,\xi] $
\item $\chi \langle \eta,\xi \rangle = \langle \nabla_\chi \eta,\xi \rangle + \langle\eta ,  \nabla_\chi \xi \rangle$,
\end{enumerate}
where $[\xi,\eta]$ is the Lie bracket, i.e., a function from the set of smooth function to itself defined by $[\xi,\eta]g = \xi(\eta (g)) - \eta(\xi (g))$. \label{def:1}
\end{definition}

For the manifolds of interest in this paper, the Lie bracket can be written as the implicit directional differentiation $[\xi,\eta] = \text{D}(\eta)[\xi] - \text{D}(\xi)[\eta]$. The expression of the Levi-Civita can be computed using the Koszul formula:
\begin{align}
2 \langle \nabla_\chi \eta,\xi \rangle &= \chi \langle \eta,\xi \rangle + \eta \langle \xi,\chi \rangle - \xi \langle \chi,\eta \rangle \nonumber \\
& \qquad - \langle \chi,[\eta,\xi] \rangle + \langle \eta,[\xi,\chi] \rangle + \langle \xi,[\chi,\eta] \rangle \label{eq:15}
\end{align}
Note that connections and particularity the Levi-Civita, are defined for all vector fields on $\mathcal{M}$. However for the purpose of this paper, only the tangent bundle is of interest. With the above notion of connection, the Riemannian Hessian can be written as:
\begin{definition}
The Riemannian Hessian of $f$ at ${\mathbf{X}}$, denoted by $\text{hess }f({\mathbf{X}})$, of a manifold $\mathcal{M}$ is a mapping from $\mathcal{T}_{\mathbf{X}}\mathcal{M}$ into itself defined by:
\begin{align}
\text{hess }f({\mathbf{X}})[\xi_{\mathbf{X}}] = \nabla_{\xi_{\mathbf{X}}} \text{grad }f({\mathbf{X}}), \ \forall \ \xi_{\mathbf{X}} \in \mathcal{T}_{\mathbf{X}}\mathcal{M},
\end{align}
where $\text{grad }f({\mathbf{X}})$ is the Riemannian gradient and $\nabla$ is the Riemannian connection on $\mathcal{M}$.
\end{definition}

It can readily be verified that the Riemannian Hessian verify similar property as the Euclidean one, i.e. for all $\xi_{\mathbf{X}},\eta_{\mathbf{X}} \in \mathcal{T}_{\mathbf{X}}\mathcal{M}$, we have 
\begin{align}
\langle \text{hess }f({\mathbf{X}})[\xi_{\mathbf{X}}],\eta_{\mathbf{X}} \rangle_{\mathbf{X}} = \langle \xi_{\mathbf{X}}, \text{hess }f({\mathbf{X}})[\eta_{\mathbf{X}}]\rangle_{\mathbf{X}} , \nonumber
\end{align}

\begin{algorithm}[t!]
\begin{algorithmic}[1]
\REQUIRE Manifold $\mathcal{M}$, function $f$, retraction $R$, and affine connection $\nabla$.
\STATE Initialize ${\mathbf{X}} \in \mathcal{M}$.
\WHILE {$||\text{grad }f({\mathbf{X}})||_{{\mathbf{X}}} \geq \epsilon$}
\STATE Find descent direction $\xi_{\mathbf{X}} \in \mathcal{T}_{\mathbf{X}}\mathcal{M}$ such that:
\begin{align}
\text{hess }f({\mathbf{X}})[\xi_{\mathbf{X}}] = -\text{grad }f({\mathbf{X}}),
\end{align}
wherein $\text{hess }f({\mathbf{X}})[\xi_{\mathbf{X}}] = \nabla_{\xi_{\mathbf{X}}} \text{grad }f({\mathbf{X}})$
\STATE Retract ${\mathbf{X}} = R_{\mathbf{X}}(\xi_{\mathbf{X}})$.
\ENDWHILE
\STATE Output ${\mathbf{X}}$.
\end{algorithmic}
\caption{Newton's method on Riemannian Manifold}
\label{alg2}
\end{algorithm}

\begin{remark}
The name of Riemannian gradient and Hessian is due to the fact that the function $f$ can be approximated in a neighborhood of ${\mathbf{X}}$ by the following:
\begin{align}
f({\mathbf{X}}+\delta {\mathbf{X}}) &= f({\mathbf{X}}) + \langle \text{grad }f({\mathbf{X}}),\text{Exp}_{\mathbf{X}}^{-1}(\delta {\mathbf{X}})\rangle_{\mathbf{X}}    \\ 
& \quad + \cfrac{1}{2} \langle \text{hess }f({\mathbf{X}})[\text{Exp}_{\mathbf{X}}^{-1}(\delta {\mathbf{X}})] ,\text{Exp}_{\mathbf{X}}^{-1}(\delta {\mathbf{X}})\rangle_{\mathbf{X}}. \nonumber
\end{align}
\end{remark}

Using the above definitions, the generalization of Newton's method to Riemannian optimization is done by replacing both the Euclidean gradient and Hessian by their Riemannian counterpart in \eref{eq:5}. Hence, the search direction is the tangent vector $\xi_{\mathbf{X}}$ that satisfies $\text{hess }f({\mathbf{X}})[\xi_{\mathbf{X}}] = -\text{grad }f({\mathbf{X}})$. The update is found by retraction the tangent vector to the manifold. The steps of the algorithm are illustrated in \algref{alg2}.

\subsection{Problems of Interest}

As shown in the previous section, computing the Riemannian gradient and Hessian for a given function over some manifold $\mathcal{M}$ allows the design of efficient algorithms that exploit the geometrical structure of the problem. The paper's main contribution is to propose a framework for solving a subset of convex programs including those in which the optimization variable represents a doubly stochastic and possibly symmetric and/or definite multidimensional probability distribution function.

In particular, the paper derives the relationship between the Euclidean gradient and Hessian and their Riemannian counterpart for the manifolds of doubly stochastic matrices, symmetric stochastic matrices, and symmetric positive stochastic matrices. In other words, for a convex function $f: \mathds{R}^{n \times m}\longrightarrow \mathds{R}$, the paper proposes solving the following problem:
\begin{subequations}
\label{eq:6}
\begin{align}
\min & \  f(\mathbf{X}) \\
\label{eq:7}  {\rm s.t.\ }  &X_{ij} > 0, \forall \ 1 \leq i \leq n, 1 \leq j \leq m,\\
\label{eq:8} & \sum_{j=1}^m X_{ij} = 1, \forall \ 1 \leq i \leq n,\\
\label{eq:9} & \sum_{i=1}^n X_{ij} = 1, \forall \ 1 \leq j \leq m,\\
\label{eq:10} & \mathbf{X} = \mathbf{X}^T, \\
\label{eq:11} & \mathbf{X} \succ \mathbf{0},
\end{align}
\end{subequations}
wherein constraints \eref{eq:7}-\eref{eq:8} produce a stochastic matrix, \eref{eq:7}-\eref{eq:9} a doubly stochastic one, \eref{eq:7}-\eref{eq:10} a symmetric stochastic one, and \eref{eq:7}-\eref{eq:11} a definite symmetric matrix. While the first scenario is studied in \cite{7182334}, the next sections study each problem, respectively. Let $\mathbf{1}$ be the all ones vector and define the multinomial, doubly stochastic multinomial, symmetric multinomial, and definite multinomial, respectively, as follows:
\begin{align*}
\mathds{P}_{n}^m &= \left\{ \mathbf{X} \in \mathds{R}^{n \times m} \big| X_{ij} > 0,\ \mathbf{X}\mathbf{1}=\mathbf{1} \right\} \\
\mathds{D}\mathds{P}_{n} &= \left\{ \mathbf{X} \in \mathds{R}^{n \times n} \big|  X_{ij} > 0,\ \mathbf{X}\mathbf{1}=\mathbf{1},\ \mathbf{X}^T\mathbf{1}=\mathbf{1} \right\} \\
\mathds{S}\mathds{P}_{n} &= \left\{ \mathbf{X} \in \mathds{R}^{n \times n} \big|  X_{ij} > 0,\ \mathbf{X}\mathbf{1}=\mathbf{1},\ \mathbf{X} = \mathbf{X}^T  \right\} \\
\mathds{S}\mathds{P}_{n}^+ &= \left\{ \mathbf{X} \in \mathds{R}^{n \times n} \big| X_{ij} > 0,\ \mathbf{X}\mathbf{1}=\mathbf{1},\ \mathbf{X} = \mathbf{X}^T,\ \mathbf{X} \succ \mathbf{0} \right\} \\
\end{align*}
For all the above manifolds, the paper uses the Fisher information as the Riemannian metric $g$ those restriction on $\mathcal{T}_\mathbf{X}\mathcal{M}$ is defined by:
\begin{align}
g(\xi_\mathbf{X},\eta_\mathbf{X}) &= \langle\xi_\mathbf{X},\eta_\mathbf{X}\rangle_\mathbf{X} = \text{Tr}((\xi_\mathbf{X} \oslash \mathbf{X})(\eta_\mathbf{X})^T )  \\
&= \sum_{i=1}^n \sum_{j=1}^m \cfrac{(\xi_\mathbf{X})_{ij} (\eta_\mathbf{X})_{ij}}{\mathbf{X}_{ij}}, \ \forall \ \xi_{\mathbf{X}},\eta_{\mathbf{X}} \in \mathcal{T}_\mathbf{X}\mathcal{M}.\nonumber
\end{align}
Endowing the multinomial with the Fisher information as Riemannian metric gives the manifold a differential structure that is invariant over the choice of coordinate system. More information about the Fisher information metric and its use in information goemetry can be found in \cite{5181555}. Using the manifold definition above, the optimization problems can be reformulated over the manifolds as:
\begin{align*}
\min_{\mathbf{X} \in \mathds{P}_{n}^m } &   f(\mathbf{X}) , \ \min_{\mathbf{X} \in \mathds{D}\mathds{P}_{n}} &   f(\mathbf{X}) , \ \min_{\mathbf{X} \in \mathds{S}\mathds{P}_{n}} &   f(\mathbf{X}) , \ 
\min_{\mathbf{X} \in \mathds{S}\mathds{P}_{n}^+} &   f(\mathbf{X}).
\end{align*}

In the rest of the paper, the notation $\mathbf{A} \oslash \mathbf{B}$ refers to the Hadamard, i.e., element-wise, division of $\mathbf{A}$ by $\mathbf{B}$. Similarly, the symbol $\odot$ denotes the Hadamard product.

\section{The Doubly Stochastic Multinomial Manifold} \label{sec:the}

This section studies the structure of the doubly stochastic manifold $\mathds{D}\mathds{P}_{n}$ and provides the expressions of the necessary ingredients to design Riemannian optimization algorithms over the manifold. 

\subsection{Manifold Geometry}

The set of doubly stochastic matrices is the set of \emph{square} matrices with positive entries such that each column and row sums to $1$. It can easily be shown that only a square matrix can verify such property. As a consequence of the Birkhoff-von Neumann theorem, $\mathds{D}\mathds{P}_{n}$ is an embedded manifold of $\mathds{R}^{n \times n}$. A short proof of the Birkhoff-von Neumann theorem using elementary geometry concepts can be found in \cite{85415218}. The dimension of $\mathds{D}\mathds{P}_{n}$ is $(n-1)^2$ which can be seen from the fact that the manifold is generated from $2n-1$ linearly independent equations specifying that the rows and columns all sums to one. The dimension of the manifold would be clearer after deriving the tangent space which is a linear space with the same dimension as the manifold.

Let $\mathbf{X} \in \mathds{D}\mathds{P}_{n}$ be a point on the manifold, the tangent space $\mathcal{T}_{\mathbf{X}}\mathds{D}\mathds{P}_{n}$ is given by the following preposition.
\begin{preposition}
The tangent space $\mathcal{T}_{\mathbf{X}}\mathds{D}\mathds{P}_{n}$ is defined by:
\begin{align}
\mathcal{T}_{\mathbf{X}}\mathds{D}\mathds{P}_{n} = \left\{\mathbf{Z} \in \mathds{R}^{n \times n} \big| \mathbf{Z}\mathbf{1}=\mathbf{0},\ \mathbf{Z}^T\mathbf{1}=\mathbf{0} \right\},
\end{align}
wherein $\mathbf{0}$ is the all zeros vector.
\label{pre1}
\end{preposition}

\begin{proof}
The technique of computing the tangent space of the manifolds of interest in this paper can be found in \appref{app1}. The complete proof of the expression of the tangent space of doubly stochastic matrices is located in the first subsection of \appref{app1}.
\end{proof}

From the expression of the tangent space, it is clear that the equations $\mathbf{Z}\mathbf{1}=\mathbf{0}$ and $\mathbf{Z}^T\mathbf{1}=\mathbf{0}$ yield only $2n-1$ linearly independent constraints as the last column constraint can be written as the sum of all rows and using the fact that the previous $(n-1)$ columns sum to zero. Let $\Pi_\mathbf{X}: \mathds{R}^{n \times n} \longrightarrow \mathcal{T}_{\mathbf{X}}\mathds{D}\mathds{P}_{n}$ be the orthogonal, in the $\langle.,.\rangle_\mathbf{X}$ sens, projection of the ambient space onto the tangent space. The expression of such operator is given in the upcoming theorem:
\begin{theorem}
The orthogonal projection $\Pi_\mathbf{X}$ has the following expression:
\begin{align}
\Pi_\mathbf{X}(\mathbf{Z}) = \mathbf{Z} - (\alpha \mathbf{1}^T + \mathbf{1} \beta^T ) \odot \mathbf{X},
\end{align}
wherein the vectors $\alpha$ and $\beta$ are obtained through the following equations:
\begin{align}
\alpha &= (\mathbf{I} - \mathbf{X}\mathbf{X}^T)^{\dagger}(\mathbf{Z} - \mathbf{X}\mathbf{Z}^T)\mathbf{1}  \\
\beta &= \mathbf{Z}^T\mathbf{1} - \mathbf{X}^{T}\alpha,
\end{align}
with $\mathbf{Y}^{\dagger}$ being the left-pseudo inverse that satisfy $\mathbf{Y}^{\dagger}\mathbf{Y}=\mathbf{I}$.
\label{th1}
\end{theorem}

\begin{proof}
Techniques for computing the orthogonal projection on the tangent space for the manifolds of interest in this paper can be found in \appref{app2}. The projection on the tangent space of the doubly stochastic matrices can be found in the first subsection of \appref{app2}.
\end{proof}

The projection $\Pi_\mathbf{X}$ is of great interest as it would allow in the next subsection to relate the Riemannian gradient and Hessian to their Euclidean equivalent. 

\begin{remark}
The above theorem gives separate expressions of $\alpha$ and $\beta$ for ease of notations in the upcoming computation of the Hessian. However, such expressions require squaring matrix $\mathbf{X}$, i.e., $\mathbf{X}\mathbf{X}^T$ which might not be numerically stable. For implementation purposes, the vectors $\alpha$ and $\beta$ are obtained as one of the solutions (typically the left-pseudo inverse) to the linear system:
\begin{align}
\begin{pmatrix}
\mathbf{Z}\mathbf{1} \\
\mathbf{Z}^T\mathbf{1}
\end{pmatrix} = 
\begin{pmatrix}
\mathbf{I} & \mathbf{X} \\
\mathbf{X}^T & \mathbf{I}
\end{pmatrix}
\begin{pmatrix}
\alpha \\
\beta
\end{pmatrix} \label{eq:257}
\end{align} 
\label{rm1}
\end{remark}

\subsection{Riemannian Gradient and Retraction Computation}

This subsection first derives the relationship between the Riemannian gradient and its Euclidean counterpart for the manifold of interest. The equation relating these two quantities is first derived in \cite{7182334} for the multinomial manifold, but no proof is provided therein. For completeness purposes, we provide the lemma with its proof in this manuscript.

\begin{lemma}
The Riemannian gradient $\text{grad }f(\mathbf{X})$ can be obtained from the Euclidean gradient $\text{Grad }f(\mathbf{X})$ using the identity:
\begin{align}
\text{grad }f(\mathbf{X}) = \Pi_\mathbf{X}(\text{Grad }f(\mathbf{X}) \odot \mathbf{X})
\end{align}
\label{lem1}
\end{lemma}

\begin{proof}
As shown in \sref{sec:opt}, the Riemannian gradient is by definition the \emph{unique} element of $\mathcal{T}_\mathbf{X} \mathds{D}\mathds{P}_{n}$ that is related to the directional derivative through the Riemannian metric as follows: 
\begin{align}
\langle \text{grad }f(\mathbf{X}), \xi_\mathbf{X} \rangle_\mathbf{X} = \text{D} f(\mathbf{X}) [\xi_\mathbf{X}],\ \forall \ \xi_\mathbf{X} \in \mathcal{T}_\mathbf{X} \mathds{D}\mathds{P}_{n}.
\label{eq:12}
\end{align}
Since the Riemannian gradient is unique, then finding an element of the tangent space that verifies the equality for all tangent vectors is sufficient to conclude that it is the Riemannian gradient. Now note that the Euclidean gradient can be written as a function of the directional derivative using the \emph{usual} scalar product as:
\begin{align}
\langle \text{Grad }f(\mathbf{X}), \xi \rangle = \text{D} f(\mathbf{X}) [\xi],\ \forall \ \xi \in \mathds{R}^{n \times n},
\end{align}
In particular, by restriction the above equation to $\mathcal{T}_\mathbf{X} \mathds{D}\mathds{P}_{n} \subset \mathds{R}^{n \times n}$ and converting the usual inner product to the Riemannian one, we can write:
\begin{align}
\langle \text{Grad }f(\mathbf{X}) , \xi_\mathbf{X} \rangle&= \langle \text{Grad }f(\mathbf{X}) \odot \mathbf{X}, \xi_\mathbf{X} \rangle_\mathbf{X} \nonumber \\
&= \text{D} f(\mathbf{X}) [\xi_\mathbf{X}],\ \forall \ \xi_\mathbf{X} \in \mathcal{T}_\mathbf{X} \mathds{D}\mathds{P}_{n}. 
\end{align}

Finally, projecting the scaled Euclidean gradient onto the tangent space and its orthogonal complement, i.e., $\text{Grad }f(\mathbf{X}) \odot \mathbf{X} = \Pi_\mathbf{X}(\text{Grad }f(\mathbf{X}) \odot \mathbf{X}) + \Pi_\mathbf{X}^\perp(\text{Grad }f(\mathbf{X}) \odot \mathbf{X})$ yields
\begin{align}
\langle \text{Grad }f(\mathbf{X}) \odot \mathbf{X}, \xi_\mathbf{X} \rangle_\mathbf{X} = \langle \Pi_\mathbf{X}(\text{Grad }f(\mathbf{X}) \odot \mathbf{X}), \xi_\mathbf{X} \rangle_\mathbf{X}, \nonumber
\end{align}
wherein, by definition of the projection on the orthogonal complement of the tangent space, the following holds
\begin{align}
\langle \Pi^\perp_\mathbf{X}(\text{Grad }f(\mathbf{X}) \odot \mathbf{X}), \xi_\mathbf{X} \rangle_\mathbf{X} = 0
\end{align}
The element $\Pi_\mathbf{X}(\text{Grad }f(\mathbf{X}) \odot \mathbf{X})$ being a tangent vector that satisfy \eref{eq:12}, we conclude that:
\begin{align}
\text{grad }f(\mathbf{X}) = \Pi_\mathbf{X}(\text{Grad }f(\mathbf{X}) \odot \mathbf{X})
\end{align}
\end{proof}

Note that the result of \lref{lem1} depends solely on the expression of the Riemannian metric and thus is applicable to the three manifolds of interest in this paper. Combining the expression of the orthogonal projection with the one of \lref{lem1}, we conclude that the Riemannian gradient has the following expression:
\begin{align}
\text{grad }f(\mathbf{X}) &= \gamma - (\alpha \mathbf{1}^T + \mathbf{1}\mathbf{1}^T \gamma - \mathbf{1} \alpha^T\mathbf{X}    )  \odot \mathbf{X} \nonumber \\
\alpha &= (\mathbf{I} - \mathbf{X}\mathbf{X}^T)^{\dagger}(\gamma-\mathbf{X}\gamma^T)\mathbf{1}  \nonumber \\
\gamma &=\text{Grad }f(\mathbf{X}) \odot \mathbf{X}. \label{eq:258}
\end{align}
For numerical stability, the term $\alpha$ is computed in a similar fashion as the procedure described in \rref{rm1} wherein $\mathbf{Z}$ is replaced by $\gamma$.

As shown in \sref{sec:opt}, one needs only to define a retraction from the tangent bundle to the manifold instead of the complex exponential map to take advantage of the optimization algorithms on the Riemannian manifolds. Among all possible retractions, one needs to derive one that have low-complexity in order to obtain efficient optimization algorithms. Therefore, the canonical choice is to exploit the linear structure of the embedding space in order to derive a retraction that does not require a projection on the manifold. Such canonical retraction is given in the following theorem:
\begin{theorem}
The mapping $R: \mathcal{T}\mathds{D}\mathds{P}_{n} \longrightarrow \mathds{D}\mathds{P}_{n} $ whose restriction $R_{\mathbf{X}}$ to $ \mathcal{T}_{\mathbf{X}}\mathds{D}\mathds{P}_{n} $ is given by:
\begin{align}
R_{\mathbf{X}}(\xi_{\mathbf{X}}) = \mathbf{X} + \xi_{\mathbf{X}},
\end{align}
represents a well-defined retraction on the doubly stochastic multinomial manifold provided that $\xi_{\mathbf{X}}$ is in the neighborhood of $\mathbf{0}_{\mathbf{X}}$, i.e., $\mathbf{X}_{ij} > - \left(\xi_{\mathbf{X}}\right)_{ij},\ 1 \leq i,j \leq n$.
\label{th4}
\end{theorem}

\begin{proof}
The proof of this theorem relies on the fact that the manifold of interest is an embedded manifold of an Euclidean space. For such manifold, one needs to find a matrix decomposition with desirable dimension and smoothness properties. The relevant theorem and techniques for computing the canonical retraction on embedded manifold are given in \appref{app3}. The proof of \thref{th4} is accomplished by extending the Sinkhorn's theorem \cite{8541523974} and can be found in the first section of \appref{app3}. 
\end{proof}

The performance of the above retraction are satisfactory as long as the optimal solution $\mathbf{X}$ does not have vanishing entries, i.e., some $\mathbf{X}_{ij}$ that approach $0$. In such situations, the update procedure results in tiny steps which compromises the convergence speed of the optimization algorithms. Although the projection on the set of doubly stochastic matrices is difficult \cite{0651518}, this paper proposes a highly efficient retraction that take advantage of the structure of both the manifold and its tangent space. Define the set of entry-wise positive matrices $\overline{\mathds{R}}^{n \times n} = \{ \mathbf{X} \in \mathds{R}^{n \times n} \ | \ \mathbf{X}_{ij} >0\}$ and let $\mathcal{P}: \overline{\mathds{R}}^{n \times n} \longrightarrow \mathds{D}\mathds{P}_{n}$ be the projection onto the set of doubly stochastic matrices obtained using the Sinkhorn-Knopp algorithm \cite{8541523974}. The proposed retraction, using the element-wise exponential of a matrix $\text{ exp}(.)$, is given in the following lemma
\begin{lemma}
The mapping $R: \mathcal{T}\mathds{D}\mathds{P}_{n} \longrightarrow \mathds{D}\mathds{P}_{n} $ whose restriction $R_{\mathbf{X}}$ to $ \mathcal{T}_{\mathbf{X}}\mathds{D}\mathds{P}_{n} $ is given by:
\begin{align}
R_{\mathbf{X}}(\xi_{\mathbf{X}}) = \mathcal{P} \left( \mathbf{X} \odot\text{exp}(\xi_{\mathbf{X}} \oslash \mathbf{X}) \right),
\end{align}
is a retraction on the doubly stochastic multinomial manifold for all $\xi_{\mathbf{X}} \in \mathcal{T}\mathds{D}\mathds{P}_{n}$.
\label{lem5}
\end{lemma}

\begin{proof}
To show that the operator represents a well-defined retraction, one needs to demonstrate that the centering and local rigidity conditions are satisfied. The fact that the mapping is obtained from the smoothness of the projection onto the set of doubly stochastic matrices which is provided in \appref{app3}. The centering property is straightforward, i.e.,:
\begin{align}
R_{\mathbf{X}}(\mathbf{0}) &= \mathcal{P} \left( \mathbf{X} \odot\text{exp}(\mathbf{0})\right) = \mathcal{P} \left( \mathbf{X}\right) = \mathbf{X},
\end{align}
wherein the last inequality is obtained from the fact that $\mathbf{X}$ is a doubly stochastic matrix. 

To prove the local rigidity condition, one needs to study the perturbation of $\mathcal{P} ( \mathbf{X})$ around a ``small" perturbation $\partial \mathbf{X}$ in the tangent space $\mathcal{T}\mathds{D}\mathds{P}_{n}$ wherein small refers to the fact that $\mathbf{X}+\partial \mathbf{X} \in \overline{\mathds{R}}^{n \times n}$. First note from the Sinkhorn-Knopp algorithm that $\mathcal{P} ( \mathbf{X})=\mathbf{D}_1\mathbf{X}\mathbf{D}_2$. However, since $\mathbf{X}$ is already doubly stochastic, then $\mathbf{D}_1=\mathbf{D}_2=\mathbf{I}$. The first order approximation of the $\mathcal{P}$ can be written as:
\begin{align}
&\mathcal{P}(\mathbf{X}+\partial \mathbf{X}) = (\mathbf{D}_1+\partial\mathbf{D}_1)(\mathbf{X}+\partial\mathbf{X})(\mathbf{D}_2+\partial\mathbf{D}_2)  \nonumber \\
&\qquad\approx \mathbf{D}_1\mathbf{X}\mathbf{D}_2 + \mathbf{D}_1\partial\mathbf{X}\mathbf{D}_2 +\partial\mathbf{D}_1\mathbf{X}\mathbf{D}_2 +\mathbf{D}_1\mathbf{X}\partial\mathbf{D}_2 \nonumber \\
&\qquad\approx \mathbf{X} +\partial\mathbf{X} +\partial\mathbf{D}_1\mathbf{X} +\mathbf{X}\partial\mathbf{D}_2
\end{align} 
Since $\mathcal{P}(\mathbf{X}+\partial \mathbf{X})$ and $\mathbf{X}$ are doubly stochastic and $\partial\mathbf{X}$ is in the tangent space, then we obtain:
\begin{align}
&\mathcal{P}(\mathbf{X}+\partial \mathbf{X}) \mathbf{1} = (\mathbf{X} +\partial\mathbf{X} +\partial\mathbf{D}_1\mathbf{X} +\mathbf{X}\partial\mathbf{D}_2)  \mathbf{1} \Rightarrow \nonumber \\
&\partial\mathbf{D}_1\mathbf{X}\mathbf{1} +\mathbf{X}\partial\mathbf{D}_2 \mathbf{1} = \partial\mathbf{D}_1\mathbf{1} +\mathbf{X}\partial\mathbf{D}_2 \mathbf{1} = \mathbf{0}
\end{align}
Similarly, by post multiplying by $\mathbf{1}^T$, we obtain $\mathbf{1}^T\partial\mathbf{D}_1\mathbf{X} +\mathbf{1}^T\partial\mathbf{D}_2= \mathbf{0}^T$. For easy of notation, let $\partial\mathbf{D}_1\mathbf{1} = \partial\mathbf{d}_1$, i.e., $\partial\mathbf{d}_1$ is the vector created from the diagonal entries of $\partial\mathbf{D}_1$ and the same for $\partial\mathbf{D}_2$. Combining both equations above, the perturbation on the diagonal matrices satisfy the condition:
\begin{align}
\begin{pmatrix}
\mathbf{I} & \mathbf{X} \\
\mathbf{X}^T & \mathbf{I}
\end{pmatrix} 
\begin{pmatrix}
\partial\mathbf{d}_1  \\
\partial\mathbf{d}_2
\end{pmatrix} 
= \begin{pmatrix}
\mathbf{0} \\
\mathbf{0}
\end{pmatrix} 
\end{align}
In other words, $\begin{pmatrix}
\partial\mathbf{d}_1  \\
\partial\mathbf{d}_2
\end{pmatrix}$ is the null space of the above matrix which is generated by $\begin{pmatrix}
\mathbf{1} \\
-\mathbf{1}
\end{pmatrix}$ from the previous analysis. As a result, $\partial\mathbf{d}_1=-\partial\mathbf{d}_2=c \mathbf{1}$ which gives $\partial\mathbf{D}_1\mathbf{X} +\mathbf{X}\partial\mathbf{D}_2 = \mathbf{0}$. Therefore, $\mathcal{P}(\mathbf{X}+\partial \mathbf{X}) \approx \mathbf{X} +\partial\mathbf{X}$. Now, consider the curve $\gamma_{\xi_{\mathbf{X}}}(\tau) = R_{\mathbf{X}}(\tau\xi_{\mathbf{X}})$. The derivative of the curve at the origin can be written as:
\begin{align}
\cfrac{d\gamma_{\xi_{\mathbf{X}}}(\tau)}{d\tau}\Big|_{\tau=0} &= \lim_{\tau \rightarrow 0} \cfrac{\gamma_{\xi_{\mathbf{X}}}(\tau)-\gamma_{\xi_{\mathbf{X}}}(0)}{\tau} \nonumber \\
&= \lim_{\tau \rightarrow 0} \cfrac{\mathcal{P} ( \mathbf{X} \odot\text{exp}(\tau\xi_{\mathbf{X}} \oslash \mathbf{X}))-\mathbf{X}}{\tau} \label{eq:lem25}
\end{align}
A first order approximation of the exponential allows to express the first term in the denominator as:
\begin{align}
\mathcal{P} ( \mathbf{X} \odot\text{exp}(\tau\xi_{\mathbf{X}} \oslash \mathbf{X})) = \mathcal{P} ( \mathbf{X} + \tau\xi_{\mathbf{X}}) = \mathbf{X} + \tau\xi_{\mathbf{X}}
\end{align}
wherein the last equality is obtained from the previous analysis. Plugging the expression in the limit expression shows the local rigidity condition, i.e.,:
\begin{align}
\cfrac{d\gamma_{\xi_{\mathbf{X}}}(\tau)}{d\tau}\Big|_{\tau=0} &= \xi_{\mathbf{X}}
\end{align}
Therefore, $R_{\mathbf{X}}(\xi_{\mathbf{X}})$ is a retraction on the doubly stochastic multinomial manifold.
\end{proof}

Note that the above retraction resembles the one proposed in \cite{7182334} (without proof) for the trivial case of the multinomial manifold. The projection on the set of doubly stochastic matrices is more involved as shown in the lemma above. Further, note that the retraction does not require the tangent vector $\xi_{\mathbf{X}}$ to be in the neighborhood of $\mathbf{X}$ as the one derived in \thref{th4}. However, it is more expensive to compute as it requires projecting the update onto the manifold.

\begin{remark}
The local rigidity condition of the retraction in \eref{eq:lem25} is particularly interesting as it shows that the canonical retraction $\mathbf{X} + \xi_{\mathbf{X}}$ is the first order approximation of the retraction $\mathcal{P} \left( \mathbf{X} \odot\text{exp}(\xi_{\mathbf{X}} \oslash \mathbf{X}) \right)$ around the origin.
\end{remark}

\subsection{Connection and Riemannian Hessian Computation}

As shown in \sref{sec:opt}, the computation of the Riemannian Hessian requires the derivation of the Levi-Civita connection $\nabla_{\eta_{\mathbf{X}}} \xi_{\mathbf{X}}$. Using the result of \cite{AbsMahSep2008}, the Levi-Civita connection of a submanifold $\mathcal{M}$ of the Euclidean space $\mathds{R}^{n \times n}$ can be obtained by projecting the Levi-Civita $\overline{\nabla}_{\eta_{\mathbf{X}}} \xi_{\mathbf{X}}$ of the embedding space onto the manifold, i.e., $\nabla_{\eta_{\mathbf{X}}} \xi_{\mathbf{X}}= \Pi_\mathbf{X}(\overline{\nabla}_{\eta_{\mathbf{X}}} \xi_{\mathbf{X}})$. From the Koszul formula \eref{eq:15}, the connection $\overline{\nabla}_{\eta_{\mathbf{X}}} \xi_{\mathbf{X}}$ on $\mathds{R}^{n \times n}$ solely depends on the Riemannian metric. In other words, the connection $\overline{\nabla}_{\eta_{\mathbf{X}}} \xi_{\mathbf{X}}$ on the embedding space is the same for all the considered manifolds in this paper. For manifolds endowed with the Fisher information as metric, the Levi-Civita connection on $\mathds{R}^{n \times n}$ is given in \cite{7182334} as follows:
\begin{proposition}
The Levi-Civita connection on the Euclidean space $\mathds{R}^{n \times n}$ endowed with the Fisher information is given by:
\begin{align}
\overline{\nabla}_{\eta_{\mathbf{X}}} \xi_{\mathbf{X}} = \text{D}(\xi_{\mathbf{X}})[\eta_{\mathbf{X}}] - \cfrac{1}{2} (\eta_{\mathbf{X}} \odot \xi_{\mathbf{X}}) \oslash \mathbf{X}
\end{align}
\end{proposition}

\begin{proof}
The Levi-Civita connection is computed in \cite{7182334} using the Koszul formula. For completeness, this short proof shows that the connection do satisfy the conditions proposed in \dref{def:1}. Since the embedding space is a Euclidean space, the Lie bracket can be written as directional derivatives as follows:
\begin{align}
[\eta_{\mathbf{X}},\xi_{\mathbf{X}}] &= \text{D}(\xi_{\mathbf{X}})[\eta_{\mathbf{X}}] -\text{D}(\eta_{\mathbf{X}})[\xi_{\mathbf{X}}] \nonumber \\
&= \text{D}(\xi_{\mathbf{X}})[\eta_{\mathbf{X}}] -\cfrac{1}{2} (\eta_{\mathbf{X}} \odot \xi_{\mathbf{X}}) \oslash \mathbf{X} \nonumber \\
&\qquad \qquad- \left( \text{D}(\eta_{\mathbf{X}})[\xi_{\mathbf{X}}] -\cfrac{1}{2} (\eta_{\mathbf{X}} \odot \xi_{\mathbf{X}}) \oslash \mathbf{X} \right) \nonumber \\
&= \overline{\nabla}_{\eta_{\mathbf{X}}} \xi_{\mathbf{X}} - \overline{\nabla}_{\xi_{\mathbf{X}}} \eta_{\mathbf{X}}
\end{align}
The second property is obtained by direct computation of the right and left hand sides in \dref{def:1}. For simplicity of the notation, all the tangent vectors $\chi_{\mathbf{X}},\eta_{\mathbf{X}},\xi_{\mathbf{X}}$ in the tangent space generated by $\mathbf{X}$ are written without the subscript. The left hand side gives:
\begin{align}
&\chi \langle \eta, \xi \rangle_{\mathbf{X}} = \text{D} (\langle \eta, \xi \rangle_{\mathbf{X}})[\chi] \nonumber \\
&= \sum_{i=1}^n \sum_{j=1}^n \text{D} \left(  \cfrac{\eta_{ij} \xi_{ij}}{\mathbf{X}_{ij}} \right)[\chi_{ij}] \nonumber \\
&= \sum_{i,j=1}^n \left( \cfrac{\eta_{ij}}{\mathbf{X}_{ij}} \text{D}  \left(  \xi_{ij}  \right) + \cfrac{ \xi_{ij}}{\mathbf{X}_{ij}} \text{D} \left(  \eta_{ij} \right) +\eta_{ij} \xi_{ij}\text{D} \left(  \cfrac{1}{\mathbf{X}_{ij}} \right)  \right) [\chi_{ij}] \nonumber \\
&=\langle \text{D}_{\chi}\eta, \xi \rangle_{\mathbf{X}} + \langle \eta, \text{D}_{\chi}\xi \rangle_{\mathbf{X}} - \sum_{i=1}^n \sum_{j=1}^n \cfrac{\eta_{ij} \xi_{ij}\chi_{ij}}{\mathbf{X}^2_{ij}}  \nonumber \\ 
&=\langle \overline{\nabla}_\chi \eta,\xi \rangle_{\mathbf{X}} + \langle\eta ,  \overline{\nabla}_\chi \xi \rangle_{\mathbf{X}}.
\end{align}
\end{proof}

Recall that the Euclidean Hessian $\text{Hess }f(\mathbf{X})[\xi_{\mathbf{X}}] = \text{D}(\text{Grad }f(\mathbf{X}))[\xi_{\mathbf{X}}]$ is defined as the directional derivative of the Euclidean gradient. Using the results above, the Riemannian Hessian can be written as a function of the Euclidean gradient and Hessian as follows:
\begin{theorem}
The Riemannian Hessian $\text{hess }f(\mathbf{X})[\xi_{\mathbf{X}}]$ can be obtained from the Euclidean gradient $\text{Grad }f(\mathbf{X})$ and the Euclidean Hessian $\text{Hess }f(\mathbf{X})[\xi_{\mathbf{X}}]$ using the identity:
\begin{align}
&\text{hess }f(\mathbf{X})[\xi_{\mathbf{X}}] = \Pi_\mathbf{X}\left( \dot{\delta} - \cfrac{1}{2}(\delta \odot \xi_{\mathbf{X}}) \oslash \mathbf{X}\right)\nonumber \\
&\alpha = \epsilon(\gamma - \mathbf{X}\gamma^T)\mathbf{1} \nonumber \\
&\beta = \gamma^T\mathbf{1} - \mathbf{X}^{T}\alpha \nonumber \\
&\gamma=\text{Grad }f(\mathbf{X}) \odot \mathbf{X} \nonumber \\
&\delta = \gamma - (\alpha \mathbf{1}^T + \mathbf{1} \beta^T ) \odot \mathbf{X} \nonumber \\
&\epsilon = (\mathbf{I} - \mathbf{X}\mathbf{X}^T)^{\dagger} \nonumber \\
&\dot{\alpha} = \left[ \dot{\epsilon}(\gamma - \mathbf{X}\gamma^T) + \epsilon(\dot{\gamma}-\xi_{\mathbf{X}}\gamma-\mathbf{X}\dot{\gamma}^T)\right]\mathbf{1} \nonumber \\
&\dot{\beta}=\dot{\gamma}^T\mathbf{1} - \xi_{\mathbf{X}}^T\alpha- \mathbf{X}^{T}\dot{\alpha} \nonumber \\
&\dot{\gamma} = \text{Hess }f(\mathbf{X})[\xi_{\mathbf{X}}]\odot \mathbf{X} + \text{Grad }f(\mathbf{X}) \odot \xi_{\mathbf{X}} \nonumber \\
&\dot{\delta} = \dot{\gamma} -(\dot{\alpha} \mathbf{1}^T + \mathbf{1} \dot{\beta}^T ) \odot \mathbf{X} - (\alpha \mathbf{1}^T + \mathbf{1} \beta^T ) \odot \xi_{\mathbf{X}} \nonumber \\
&\dot{\epsilon} = \epsilon(\mathbf{X}\xi_{\mathbf{X}}^T + \xi_{\mathbf{X}} \mathbf{X}^T)\epsilon
\end{align}
\label{th6}
\end{theorem}

\begin{proof}
The useful results to compute the Riemannian Hessian for the manifold of interest in this paper can be found in \appref{app4}. The expression of the Riemannian Hessian for the doubly stochastic multinomial manifold can be found in the first subsection of the appendix.
\end{proof}

\section{The Symmetric Multinomial Manifold} \label{sec:the2}

Whereas the doubly stochastic multinomial manifold is regarded as an embedded manifold of the vector space of matrices $\mathds{R}^{n \times n}$, the symmetric and positive multinomial manifolds are seen as embedded manifolds of the set of symmetric matrices. In other words, the embedding Euclidean space is the space of symmetric matrices $\mathcal{S}_n$ defined as:
\begin{align}
\mathcal{S}_n &= \left\{ \mathbf{X} \in \mathds{R}^{n \times n} \big| \mathbf{X} = \mathbf{X}^T  \right\} 
\end{align}

Such choice of the ambient space allows to reduce the ambient dimension from $n^2$ to $\frac{n(n+1)}{2}$ and thus enables the simplification of the projection operators. As a result, the expression of Riemannian gradient and Hessian can be computed more efficiently.

\subsection{Manifold Geometry, Gradient, and Retraction}

Let $\mathbf{X} \in \mathds{S}\mathds{P}_{n}$ be a point on the manifold, the tangent space $\mathcal{T}_{\mathbf{X}}\mathds{S}\mathds{P}_{n}$ is given by the following preposition.
\begin{preposition}
The tangent space $\mathcal{T}_{\mathbf{X}}\mathds{S}\mathds{P}_{n}$ is defined by:
\begin{align}
\mathcal{T}_{\mathbf{X}}\mathds{S}\mathds{P}_{n} = \left\{\mathbf{Z} \in \mathcal{S}_n \big| \mathbf{Z}\mathbf{1}=\mathbf{0}\right\}.
\end{align}
\label{pre2}
\end{preposition}

\begin{proof}
The technique of computing the tangent space of manifold of interest in this paper can be found in \appref{app1}. The complete proof of the expression of the tangent space of symmetric stochastic matrices is located in the second subsection of \appref{app1}.
\end{proof}

Let $\Pi_\mathbf{X}: \mathcal{S}_n \longrightarrow \mathcal{T}_{\mathbf{X}}\mathds{S}\mathds{P}_{n}$ be the orthogonal, in the $\langle.,.\rangle_\mathbf{X}$ sens, projection of the ambient space onto the tangent space. Note that the ambient space for the symmetric multinomial $\mathds{S}\mathds{P}_{n}$ is the set of symmetric matrices $\mathcal{S}_n$ and not the set of all matrices $\mathds{R}^{n \times n}$ as in \sref{sec:the}. The following theorem gives the expression of the projection operator:
\begin{theorem}
The orthogonal projection $\Pi_\mathbf{X}$ operator onto the tangent set has the following expression
\begin{align}
\Pi_\mathbf{X}(\mathbf{Z}) = \mathbf{Z} - (\alpha \mathbf{1}^T + \mathbf{1} \alpha^T ) \odot \mathbf{X},
\end{align}
wherein the vector $\alpha$ is computed as:
\begin{align}
\alpha &= (\mathbf{I} + \mathbf{X})^{-1}\mathbf{Z}\mathbf{1}.
\end{align}
\label{th2}
\end{theorem}

\begin{proof}
Techniques for computing the orthogonal projection on the tangent space for the manifolds of interest in this paper can be found in \appref{app2}. The projection on the tangent space of the doubly stochastic matrices can be found in the second subsection of \appref{app2} and its derivation from the projection of the doubly stochastic manifold in the third subsection.
\end{proof}

Using the result of \lref{lem1} and using the expression of the projection onto the tangent space, the Rienmannian gradient can be efficiently computed as:
\begin{align}
\text{grad }f(\mathbf{X}) &= \gamma - (\alpha \mathbf{1}^T + \mathbf{1} \alpha^T) \odot \mathbf{X} \nonumber \\
\alpha &= (\mathbf{I} + \mathbf{X})^{-1}\gamma\mathbf{1} \nonumber \\
\gamma &= (\text{Grad }f(\mathbf{X}) \odot \mathbf{X}),
\end{align}
where $\gamma$ is a simple sum that can be computed efficiently.

Similar to the result for the doubly stochastic multinomial manifold, the canonical retraction on the symmetric multinomial manifold can be efficiently computed as shown in the following corollary.
\begin{corollary}
The mapping $R: \mathcal{T}\mathds{S}\mathds{P}_{n} \longrightarrow \mathds{S}\mathds{P}_{n} $ whose restriction $R_{\mathbf{X}}$ to $ \mathcal{T}_{\mathbf{X}}\mathds{S}\mathds{P}_{n} $ is given by:
\begin{align}
R_{\mathbf{X}}(\xi_{\mathbf{X}}) = \mathbf{X} + \xi_{\mathbf{X}},
\end{align}
represents a well-defined retraction on the symmetric multinomial manifold provided that $\xi_{\mathbf{X}}$ is in the neighborhood of $\mathbf{0}_{\mathbf{X}}$, i.e., $\mathbf{X}_{ij} > - \left(\xi_{\mathbf{X}}\right)_{ij},\ 1 \leq i,j \leq n$.
\label{cor1}
\end{corollary}

\begin{proof}
The proof of this corollary follows similar steps like the one for the doubly stochastic multinomial manifold by considering that the manifold is embedded in an Euclidean subspace. Techniques for computing the retraction on embedded manifold is given in \appref{app3}. Note that the result of the doubly stochastic multinomial is not directly applicable as the embedding space is different ($\mathcal{S}_n$ instead of $\mathds{R}^{n \times n}$). The problem is solved by using the DAD theorem \cite{CSIMA1972147} instead of the Sinkhorn's one \cite{8541523974}. The complete proof of the corollary is given in the second subsection of \appref{app3}. 
\end{proof}

The canonical retraction suffers from the same limitation as the one discussed in the previous section. Indeed, the performance of the optimization algorithm heavily depend on whether the optimal solution has vanishing entries or not. This section shows that the retraction proposed in \lref{lem5} is a valid retraction on the set of symmetric double stochastic matrices. However, instead of the Sinkhorn-Knopp algorithm \cite{8541523974}, this part uses the DAD algorithm \cite{CSIMA1972147} to project the retracted vector. Let $\overline{\mathcal{S}}_n = \left\{ \mathbf{X} \in \mathds{R}^{n \times n} \big|   X_{ij} > 0,\ \mathbf{X} = \mathbf{X}^T  \right\}$ represent the set of symmetric, element-wise positive matrices. The projection onto the set of symmetric doubly stochastic matrices is denoted by the operator $\mathcal{P}^+:\overline{\mathcal{S}}_n \longrightarrow \mathds{S}\mathds{P}_{n}$. The retraction is given in the following corollary.
\begin{corollary}
The mapping $R: \mathcal{T}\mathds{S}\mathds{P}_{n} \longrightarrow \mathds{S}\mathds{P}_{n} $ whose restriction $R_{\mathbf{X}}$ to $ \mathcal{T}_{\mathbf{X}}\mathds{S}\mathds{P}_{n} $ is given by:
\begin{align}
R_{\mathbf{X}}(\xi_{\mathbf{X}}) = \mathcal{P}^+ \left( \mathbf{X} \odot\text{exp}(\xi_{\mathbf{X}} \oslash \mathbf{X}) \right),
\end{align}
is a retraction on the symmetric doubly stochastic multinomial manifold for all $\xi_{\mathbf{X}} \in \mathcal{T}\mathds{S}\mathds{P}_{n}$.
\end{corollary}

\begin{proof}
The proof of this corollary is straightforward. Indeed, after showing that the range space of $R_{\mathbf{X}}(\xi_{\mathbf{X}})$ is the set symmetric element-wise matrices $\overline{\mathcal{S}}_n$, the proof concerning the centering and local rigidity of the retraction are similar to the one in \lref{lem5}.
\end{proof}

\subsection{Connection and Riemannian Hessian Computation}

As discussed earlier, the Levi-Civita connection solely depends on the Riemannian metric. Therefore, the symmetric stochastic multinomial manifold shares the same retraction on the embedding space\footnote{Even though the embedding spaces are not the same for both manifolds, one can easily show that the expression of the connection is invariant.} as the doubly stochastic multinomial manifold. The Riemannian Hessian can be obtained by differentiating the Riemnanian gradient using the projection of the Levi-Civita connection onto the manifold as shown in the below corollary:
\begin{corollary}
The Riemannian Hessian $\text{hess }f(\mathbf{X})[\xi_{\mathbf{X}}]$ can be obtained from the Euclidean gradient $\text{Grad }f(\mathbf{X})$ and the Euclidean Hessian $\text{Hess }f(\mathbf{X})[\xi_{\mathbf{X}}]$ using the identity:
\begin{align}
&\text{hess }f(\mathbf{X})[\xi_{\mathbf{X}}] = \Pi_\mathbf{X}\left( \dot{\delta} - \cfrac{1}{2}(\delta \odot \xi_{\mathbf{X}}) \oslash \mathbf{X}\right) \nonumber \\
&\alpha = (\mathbf{I} + \mathbf{X})^{-1}\gamma\mathbf{1} \nonumber \\
&\delta = \gamma - (\alpha \mathbf{1}^T + \mathbf{1} \alpha^T ) \odot \mathbf{X} \nonumber \\
&\gamma=\text{Grad }f(\mathbf{X}) \odot \mathbf{X} \nonumber \\
&\dot{\alpha} = \left((\mathbf{I} + \mathbf{X})^{-1}\dot{\gamma} -(\mathbf{I} + \mathbf{X})^{-1}\xi_{\mathbf{X}}(\mathbf{I} + \mathbf{X})^{-1}\gamma\right)\mathbf{1} \nonumber \\
&\dot{\delta} = \dot{\gamma} -(\dot{\alpha} \mathbf{1}^T + \mathbf{1} \dot{\alpha}^T ) \odot \mathbf{X} - (\alpha \mathbf{1}^T + \mathbf{1} \alpha^T ) \odot \xi_{\mathbf{X}} \nonumber \\
&\dot{\gamma} = \text{Hess }f(\mathbf{X})[\xi_{\mathbf{X}}]\odot \mathbf{X} + \text{Grad }f(\mathbf{X}) \odot \xi_{\mathbf{X}}  
\end{align}
\label{cor3}
\end{corollary}

\section{Extension to the Definite Symmetric Multinomial Manifold} \label{sec:ext}

The definite symmetric stochastic multinomial manifold is defined as the subset of the symmetric stochastic multinomial manifold wherein the matrix of interest is positive definite. Similar to the condition $\mathbf{X}_{ij} > 0$, the strict condition $\mathbf{X} \succ 0$, i.e., full-rank matrix, ensures that the manifold has a differentiable structure. 

The positive-definiteness constraint is a difficult one to retract. In order to produce highly efficient algorithms, one usually needs a re-parameterization of the manifold and to regard the new structure as a quotient manifold, e.g., a Grassmann manifold. However, this falls outside the scope of this paper and is left for future investigation. This part extends the previous study and regards the manifold as an embedded manifold of $\mathcal{S}_n$ for which two retractions are proposed.  

\subsection{Manifold Geometry}

The manifold geometry of the definite symmetric stochastic multinomial manifold is similar to the one of the previous manifold. Indeed, the strictly positive constraint being an inequality one, the dimension of the manifold is unchanged. Therefore, let $\mathbf{X} \in \mathds{S}\mathds{P}^+_{n}$ be a point on the manifold, the tangent space $\mathcal{T}_{\mathbf{X}}\mathds{S}\mathds{P}^+_{n}$ of the definite symmetric stochastic multinomial manifold is similar to tangent space of the symmetric stochastic multinomial manifold, i.e.,:
\begin{align}
\mathcal{T}_{\mathbf{X}}\mathds{S}\mathds{P}^+_{n} = \left\{\mathbf{Z} \in \mathcal{S}_n \big| \mathbf{Z}\mathbf{1}=\mathbf{0}\right\}.
\end{align}

As a result, the expression of the Riemannian gradient and Hessian are the same as the one presented in the previous section. Hence, one only needs to design a retraction to derive optimization algorithms on the manifold.

\subsection{Retraction on the Cone of Positive Definite Matrices}

As shown in the previous subsection, the geometry of the definite symmetric multinomial manifold is similar to the symmetric multinomial manifold one. Therefore, one can extend the canonical retraction proposed in the previous section. However, even tough the retraction looks similar to the one proposed in the previous section, its implementation is more problematic as it includes a condition on the eigenvalues. Hence, the section proposes another retraction that exploits the definite structure of the manifold and uses the matrix exponential to retract the tangent vectors as shown in the following theorem.
\begin{theorem}
Define the map $R_{\mathbf{X}}$ from $ \mathcal{T}_{\mathbf{X}}\mathds{S}\mathds{P}^+_{n} $ to $\mathds{S}\mathds{P}^+_{n} $ by:
\begin{align}
R_{\mathbf{X}}(\xi_{\mathbf{X}}) = \mathbf{X} + \dfrac{1}{\omega_{\mathbf{X}}} \mathbf{I} - \dfrac{1}{\omega_{\mathbf{X}}} \mathbf{e}^{-\omega_{\mathbf{X}}\xi_{\mathbf{X}}},
\end{align}
wherein $\mathbf{e}^{\mathbf{Y}}$ the matrix exponential\footnote{Not to be confused with the exponential map $\text{Exp}$ of the Riemannian geometry or with the element-wise exponential $\text{exp}$.} of matrix $\mathbf{Y}$ and $\omega_{\mathbf{X}}$ is a scalar that ensures:
\begin{align}
R_{\mathbf{X}}(\xi_{\mathbf{X}})_{ij} &> 0 ,\  1 \leq i,j \leq n \label{eq:13}\\
R_{\mathbf{X}}(\xi_{\mathbf{X}}) &\succ \mathbf{0}, \label{eq:14}
\end{align}
for all $\xi_{\mathbf{X}} \in \mathcal{T}_{\mathbf{X}}\mathds{S}\mathds{P}^+_{n}$ in the neighborhood of $\mathbf{0}_\mathbf{X} $, i.e., $||\xi_{\mathbf{X}}||_F \leq \epsilon$ for some $\epsilon>0$. Then, there exists a sequence of scalars $\{\omega_{\mathbf{X}}\}_{\mathbf{X} \in \mathds{S}\mathds{P}^+_{n}}$ such that the mapping $R: \mathcal{T}\mathds{S}\mathds{P}^+_{n} \longrightarrow \mathds{S}\mathds{P}^+_{n} $, whose restriction $R_{\mathbf{X}}$ to $ \mathcal{T}_{\mathbf{X}}\mathds{S}\mathds{P}^+_{n} $, is a retraction on the definite symmetric multinomial manifold.
\label{th5}
\end{theorem}

\begin{proof}
Unlike the previous retractions that rely on the Euclidean structure of the embedding space, this retraction is obtained by direct computation of the properties of the retraction given in \sref{sec:opt}. The organization of the proof is the following: First, assuming the existence of $\omega_{\mathbf{X}}$, we show that the range of the mapping $R_{\mathbf{X}}$ is included in the definite symmetric multinomial manifold. Afterward, we demonstrate that the operator satisfies the centering and local rigidity conditions. Therefore, the operator represents a retraction. Finally, showing the existence of the scalar $\omega_{\mathbf{X}}$ for an arbitrary $\mathbf{X} \in \mathds{S}\mathds{P}^+_{n}$ concludes the proof.

Recall that the matrix exponential of a symmetric real matrix $\xi_{\mathbf{X}}$ with eigenvalue decomposition  $\xi_{\mathbf{X}}=\mathbf{U}\Lambda\mathbf{U}^T$ is given by $\mathbf{e}^{\xi_{\mathbf{X}}}=\mathbf{U}\exp(\Lambda)\mathbf{U}^T$, where $\exp(\Lambda)$ is the usual element-wise exponential of the element on the diagonal and zeros elsewhere. From the derivation of the tangent space of the definite symmetric multinomial manifold $\mathcal{T}_{\mathbf{X}}\mathds{S}\mathds{P}^+_{n}$, we have $\xi_{\mathbf{X}}\mathbf{1}=\mathbf{0}$. Therefore, $\xi_{\mathbf{X}}$ has an eigenvalue of $0$ corresponding to the eigenvector $\mathbf{1}$. As stated by the definition of the matrix exponential, the eigenvalue are exponentiated and the eigenvectors are unchanged. Therefore, $\mathbf{e}^{\xi_{\mathbf{X}}}$ (and thus $\mathbf{e}^{-\omega_{\mathbf{X}}\xi_{\mathbf{X}}}$) has an eigenvalue of $\exp(0)=1$ corresponding to the eigenvector $\mathbf{1}$, i.e., $\mathbf{e}^{-\omega_{\mathbf{X}}\xi_{\mathbf{X}}}\mathbf{1}=\mathbf{1}$. First note from the first condition on $\omega_{\mathbf{X}}$, all entries are positive. Now, computing the rows summation gives:
\begin{align}
R_{\mathbf{X}}(\xi_{\mathbf{X}})\mathbf{1} &= \mathbf{X}\mathbf{1} + \dfrac{1}{\omega_{\mathbf{X}}} \mathbf{I}\mathbf{1} - \dfrac{1}{\omega_{\mathbf{X}}} \mathbf{e}^{-\omega_{\mathbf{X}}\xi_{\mathbf{X}}}\mathbf{1} \nonumber \\
&= \mathbf{1} + \dfrac{1}{\omega_{\mathbf{X}}}\mathbf{1} - \dfrac{1}{\omega_{\mathbf{X}}} \mathbf{1} \nonumber \\
&= \mathbf{1}.
\end{align}
Hence $R_{\mathbf{X}}(\xi_{\mathbf{X}})$ is stochastic. Besides, all matrices are symmetric which concludes that the matrix is doubly stochastic. Finally, the second condition on $\omega_{\mathbf{X}}$ ensure the definiteness of the matrix which concludes that $R_{\mathbf{X}}(\xi_{\mathbf{X}}) \in \mathds{S}\mathds{P}^+_{n}$.

The centering property can be easily checked by evaluating the retraction $R_{\mathbf{X}}$ at the zero-element $\mathbf{0}_\mathbf{X} $ of $ \mathcal{T}_{\mathbf{X}}\mathds{S}\mathds{P}^+_{n} $. Indeed, we obtain:
\begin{align}
R_{\mathbf{X}}(\mathbf{0}_\mathbf{X}) &= \mathbf{X} + \dfrac{1}{\omega_{\mathbf{X}}} \mathbf{I} - \dfrac{1}{\omega_{\mathbf{X}}} \mathbf{e}^{-\omega_{\mathbf{X}}\mathbf{0}_\mathbf{X}} \nonumber \\
&= \mathbf{X} + \dfrac{1}{\omega_{\mathbf{X}}} \mathbf{I} - \dfrac{1}{\omega_{\mathbf{X}}} \mathbf{I} = \mathbf{X}.
\end{align}
The speed of the rigidity curve $\gamma_{\xi_{\mathbf{X}}}(\tau) = R_{\mathbf{X}}(\tau \xi_{\mathbf{X}})$ at the origin is given by:
\begin{align}
\cfrac{d\gamma_{\xi_x}(\tau)}{d\tau}\Big|_{\tau=0} &= - \dfrac{1}{\omega_{\mathbf{X}}} \cfrac{d\mathbf{e}^{-\omega_{\mathbf{X}}\tau\xi_{\mathbf{X}}}}{d\tau} \Big|_{\tau=0} \nonumber \\
&= - \dfrac{1}{\omega_{\mathbf{X}}} \mathbf{U}\cfrac{d\exp(-\omega_{\mathbf{X}}\tau\Lambda)}{d\tau}\mathbf{U}^T\Big|_{\tau=0} \nonumber \\
&= \mathbf{U}\Lambda \exp(-\omega_{\mathbf{X}}\tau\Lambda)\mathbf{U}^T\Big|_{\tau=0} \nonumber \\
&= \mathbf{U}\Lambda \mathbf{U}\mathbf{U}^T\exp(-\omega_{\mathbf{X}}\tau\Lambda)\mathbf{U}^T\Big|_{\tau=0} \nonumber \\
&= \xi_{\mathbf{X}}\mathbf{e}^{-\omega_{\mathbf{X}}\tau\xi_{\mathbf{X}}}\Big|_{\tau=0} = \xi_{\mathbf{X}} 
\end{align}
Therefore, we conclude that $R_{\mathbf{X}}(\xi_{\mathbf{X}})$ is a well-defined retraction.

Finally, the existence of the weight $\omega_{\mathbf{X}}$ is ensured by the fact that $\overline{\mathcal{S}}_n^+$ is an open subset of $\mathcal{S}_n$. Consider a positive sequence $\{\omega^m_{\mathbf{X}}\}_{m =1}^{\infty}$ decreasing to $0$ and construct the function series $\{\mathbf{X}_m\}_{m =1}^{\infty}$ as follows:
\begin{align}
\mathbf{X}_m(\xi_{\mathbf{X}}) = \mathbf{X} + \dfrac{1}{\omega^m_{\mathbf{X}}} \mathbf{I} - \dfrac{1}{\omega^m_{\mathbf{X}}} \mathbf{e}^{-\omega^m_{\mathbf{X}}\xi_{\mathbf{X}}}.
\end{align}
We aim to show that $\{\mathbf{X}_m\}_{m =1}^{\infty}$ uniformly converges to the constant $\mathbf{X} \in \overline{\mathcal{S}}_n^+$. Since $\overline{\mathcal{S}}_n^+$ is an open set, then there exists an index $m_0$ above which, i.e., $\forall \ m \geq m_0$ the sequence $\mathbf{X}_m(\xi_{\mathbf{X}}) \in \overline{\mathcal{S}}_n^+, \forall \ \xi_{\mathbf{X}} \in \mathcal{T}_{\mathbf{X}}\mathds{S}\mathds{P}^+_{n}$ with $||\xi_{\mathbf{X}}||_F \leq \epsilon$. Hence $\omega_{\mathbf{X}}$ can be chosen to be any $\omega^m_{\mathbf{X}}$ for $ m \geq m_0$. The uniform convergence of the series is given in the following lemma
\begin{lemma}
The uniform convergence of the function series $\{\mathbf{X}_m\}_{m =1}^{\infty}$ is satisfied as $\forall \ \epsilon^\prime > 0$, $\exists \  M_0$ such that $\forall \ m \geq M_0$ the following holds:
\begin{align}
||\mathbf{X}_m(\xi_{\mathbf{X}})-\mathbf{X}||_F < \epsilon^\prime, \ \forall \ \xi_{\mathbf{X}} \in \mathcal{T}_{\mathbf{X}}\mathds{S}\mathds{P}^+_{n} \text{ with } ||\xi_{\mathbf{X}}||_F \leq \epsilon \nonumber
\end{align}
\end{lemma}
\begin{proof}
Note that the condition over all tangent vectors can be replaced by the following condition (up to an abuse of notation with the $\epsilon^\prime$):
\begin{align}
||\mathbf{X}_m(&\xi_{\mathbf{X}})-\mathbf{X}||_F < \epsilon^\prime, \ \forall \ \xi_{\mathbf{X}} \in \mathcal{T}_{\mathbf{X}}\mathds{S}\mathds{P}^+_{n} \text{ with } ||\xi_{\mathbf{X}}||_F \leq \epsilon \nonumber \\
&\Leftrightarrow \sup_{\substack{\xi_{\mathbf{X}} \in \mathcal{T}_{\mathbf{X}}\mathds{S}\mathds{P}^+_{n} \\ ||\xi_{\mathbf{X}}||_F \leq \epsilon}} ||\mathbf{X}_m(\xi_{\mathbf{X}})-\mathbf{X}||_F^2 < \epsilon^\prime \nonumber \\ 
& \qquad \Leftrightarrow \max_{\substack{\xi_{\mathbf{X}} \in \mathcal{T}_{\mathbf{X}}\mathds{S}\mathds{P}^+_{n} \\ ||\xi_{\mathbf{X}}||_F \leq \epsilon}} ||\mathbf{X}_m(\xi_{\mathbf{X}})-\mathbf{X}||_F^2 < \epsilon^\prime,
\end{align}
wherein the last equivalence is obtained from the fact that the search space is closed. The last expression allows us to work with an upper bound of the distance. Indeed, the distance can be bound by:
\begin{align}
||\mathbf{X}_m(\xi_{\mathbf{X}})-\mathbf{X}||_F^2 &= \dfrac{1}{(\omega^m_{\mathbf{X}})^2}|| \mathbf{I} - \mathbf{e}^{-\omega^m_{\mathbf{X}}\xi_{\mathbf{X}}} ||_F^2\nonumber \\
 &= \dfrac{1}{(\omega^m_{\mathbf{X}})^2} \sum_{i=1}^n(1-\exp(-\omega^m_{\mathbf{X}}\lambda_i))^2 \nonumber \\
 &\leq \dfrac{n}{(\omega^m_{\mathbf{X}})^2} (1-\exp(-\omega^m_{\mathbf{X}}\epsilon))^2,
\end{align}
with the last inequality being obtained from $||\xi_{\mathbf{X}}||_F \leq \epsilon \Rightarrow  \lambda_i\leq \epsilon , 1 \leq i \leq n$. Now using the fact that $\{\omega^m_{\mathbf{X}}\}_{m =1}^{\infty}$ is decreasing to $0$, then there exists $M_0$ such that $\forall \ m \geq M_0$, the following is true:
\begin{align}
\dfrac{n}{(\omega^m_{\mathbf{X}})^2} (1-\exp(-\omega^m_{\mathbf{X}}\epsilon))^2  \leq \epsilon \leq \epsilon^\prime
\end{align}
Combining the above results, we find out that $\forall \ \epsilon^\prime > 0$, $\exists \  M_0$ such that $\forall \ m \geq M_0$ the following holds:
\begin{align}
||\mathbf{X}_m(\xi_{\mathbf{X}})-\mathbf{X}||_F < \epsilon^\prime, \ \forall \ \xi_{\mathbf{X}} \in \mathcal{T}_{\mathbf{X}}\mathds{S}\mathds{P}^+_{n} \text{ with } ||\xi_{\mathbf{X}}||_F \leq \epsilon \nonumber
\end{align}
\end{proof}
Finally, as stated earlier, combining the uniform convergence and the fact that $\overline{\mathcal{S}}_n^+$ is an open subset of $\mathcal{S}_n$ allows us to conclude the existence of $\omega_{\mathbf{X}}$ such that both conditions \eref{eq:13} and \eref{eq:14} are satisfied for all tangent vectors $\xi_{\mathbf{X}} \in \mathcal{T}_{\mathbf{X}}\mathds{S}\mathds{P}^+_{n}$ with $||\xi_{\mathbf{X}}||_F \leq \epsilon$.
\end{proof} 

\begin{remark}
The retraction developed in \thref{th5} can be applied to the symmetric multinomial manifold in which the scalar $\omega_{\mathbf{X}}$ is chosen in such fashion so as to satisfy condition \eref{eq:13} independently of condition \eref{eq:14}. However, the retraction is not valid for the doubly stochastic multinomial manifold. Indeed, the tangent vectors of the doubly stochastic multinomial manifold being not necessarily symmetric, nothing can be claimed about the exponential of the vector $\xi_{\mathbf{X}}$.
\end{remark}

While the retraction proposed in \thref{th5} is superior to the canonical one, it still suffers from the scaling of $\xi_{\mathbf{X}}$. In other words, if the optimal solution has vanishing entries or vanishing eigenvalues, the retraction results in tiny updates which compromise the convergence speed of the proposed algorithms. As stated at the beginning of the section, the problem can be solved by re-parameterizing the manifold which falls outside the scope of this paper.

\section{Optimization Algorithms and Complexity} \label{sec:opt2}

\begin{table*}
\centering
\caption{Complexity of the steepest descent and Newton's method algorithms for the proposed manifolds.}
\begin{tabular}{|c||c|c|}
\hline
Manifold & Steepest descent algorithm & Newton's method algorithm \\
\hline \hline
Doubly Stochastic Multinomial $\mathds{D}\mathds{P}_{n}$ & $(16/3)n^3+7n^2+\log(n)\sqrt{n}$ & $32/3n^3+15n^2+\log(n)\sqrt{n}$\\
\hline 
Symmetric Multinomial $\mathds{S}\mathds{P}_{n}$ & $(1/3)n^3 + 2n^2+2n+\log(n)\sqrt{n}$ & $n^3+8n^2+17/2n+\log(n)\sqrt{n}$\\
\hline 
Definite Symmetric Multinomial $\mathds{S}\mathds{P}_{n}^+$ & $n^3+3n^2+3n$ & $4/3n^3+13/2n^2+7n$\\\hline
\end{tabular}
\label{tab:1}
\end{table*}

This section analysis the complexity of the steepest descent, summarized in \algref{alg1}, and the Newton's method, summed up in \algref{alg2}, algorithms on the proposed Riemannian Manifolds. While this section only presents the simplest first and second order algorithms, the simulation section uses the more sophisticated conjugate gradient and trust regions methods as first and second order algorithms to obtain the curves. The complexity analysis of these algorithms is similar to the one presented herein as it is entirely governed by the complexity of computing the Riemannian gradient and Hessian. \tref{tab:1} summarizes the first and second order complexity for each manifold.

\subsection{Gradient Descent Complexity}

The complexity of computing the gradient \eref{eq:258} of the doubly stochastic multinomial manifold can be decomposed into the complexity of computing $\gamma$, $\alpha$, and $\text{grad }f(\mathbf{X})$. The term $\gamma$ is a simple Hamadard product that can be computed in $n^2$ operations. The term $\alpha$ is obtained by solving the system of equations in \eref{eq:257} which takes $(2/3)(2n)^3$ when using an LU factorization. Finally, the expression of $\text{grad }f(\mathbf{X})$ requires a couple of additions and an hadamard product which can be done in $3n^2$ operations. Finally, the complexity of computing the retraction can be decomposed into the complexity of computing the update vector and the complexity of the projection. The updated vector is an hadamard product and division that can be computed in $3n^2$. The complexity of projecting a matrix $\mathbf{A}$ onto the set of doubly stochastic manifold \cite{Kalantari} with accuracy $\epsilon$ is given by:
\begin{align}
\mathcal{O}((1/\epsilon + \log(n))\sqrt{n}V/v),
\end{align}
wherein $V = \max(\mathbf{A})$ and $v=\min(\mathbf{A})$. Therefore, the total complexity of an iteration of the gradient descent algorithm on the doubly stochastic multinomial manifold is $(16/3)n^3+7n^2+\log(n)\sqrt{n}$.

The complexity of the symmetric stochastic manifold can be obtained in a similar manner. Due to the symmetry, term $\gamma$ only requires $n(n+1)/2$ operations. The term $\alpha$ is the solution to an $n\times n$ system of equations which can be solved in $(1/3)n^3$. Similarly, $\text{grad }f(\mathbf{X})$ requires $3/2 n(n+1)$. Therefore, the total complexity can be written as:
\begin{align}
(1/3)n^3 + 2n^2+2n+\log(n)\sqrt{n}.
\end{align}

The retraction on the cone of positive matrices requires $n^3+2n^2$ which gives a total complexity of the algorithm for the positive symmetric multinomial manifold of $n^3+3n^2+3n$.

\subsection{Newton's Method Complexity}

The complexity of computing the Newton's method requires computing the Riemannian gradient and Hessian and solving an $n \times n$ linear system. However, from the expressions of the Riemannian Hessian, one can note that the complexity of computing the Riemannian gradient in included in the one of the Riemannian Hessian.

For the doubly stochastic manifold, the complexity of computing the Riemannian Hessian is controlled by the complexity of the projection and the inversions. The projection onto the tangent space requires solving a $n \times n$ system and a couple of additions and an Hadamard product. The total cost of the operation is $2/3(2n)^3 + 3n^2$. The $\epsilon$ and $\dot{epsilon}$ terms are inversions and matrices products that require $4n^3$. The other terms combined require $9n^2$ operation. The retraction costs $3n^2+\log(n)\sqrt{n}$ and solving for the search direction requires $2/3(2n)^3$ which gives a total complexity of:
\begin{align}
32/3n^3+15n^2+\log(n)\sqrt{n}.
\end{align}

A similar analysis as the one above allows to conclude that the total complexity of a second order method on the symmetric and positive doubly stochastic manifold require, respectively, the following number of iterations:
\begin{align}
n^3+8n^2+17/2n+\log(n)\sqrt{n} \\
4/3n^3+13/2n^2+7n.
\end{align}

\section{Simulation Results} \label{sec:sim}

This section attests the performance of the proposed framework in efficiently solving optimization problems in which the optimization variables is a doubly stochastic matrix. The experiments are carried out using Matlab on an Intel Xeon Processor E5-1650 v4 (15M Cache, 3.60 GHz) computer with 32GB 2.4 GHz DDR4 RAM. The optimization is performed using the Matlab toolbox ``Manopt" \cite{manopt} and the conjugate gradient (denoted by CG) and trust regions (denoted by TR) algorithms.

The section is divided into three subsections. The first subsection tests the performance of the proposed solution against a popular generic solver ``CVX" \cite{cvx} for each of the manifolds. The section further shows the convergence of each manifold in reaching the same solution by mean of regularization. The second subsection solves a convex clustering problem \cite{7541267} and testifies the efficiency of the proposed algorithm against a generic solver. Finally, the last subsection shows that the proposed framework outperforms a specialized algorithm \cite{Yang:2012:CLD:3042573.3042666} in finding the solution of a non-convex clustering problem. 

\subsection{Performance of the Proposed Manifolds}

This section solves the following optimization problem:
\begin{align}
\min_{\mathbf{X} \in \mathcal{M}} ||\mathbf{A}-\mathbf{X}||_F^2,
\label{eq:545}
\end{align}
wherein the manifold $\mathcal{M}$ is the doubly stochastic, symmetric, and definite stochastic multinomial manifold, respectively. For each of the experiment, matrix $\mathbf{A}$ is generated by $\mathbf{A} = \mathbf{M} + \mathbf{N}$ with $\mathbf{M} \in \mathcal{M}$ belonging to the manifold of interest and $\mathbf{N}$ is a zero-mean white Gaussian noise. 

The optimization problem is first solved using the toolbox CVX to obtain the optimal solution $\mathbf{X}^*$ with a predefined precision. The proposed algorithms are iterated until the desired optimal solution $\mathbf{X}^*$ is reached with the same precision and the total execution time is displayed in the corresponding table.

\begin{table}[t]
\begin{small}
\centering
\begin{tabular}{|c||c|c|c|c|c|}
\hline  
$n$&60&70&80&90&100  \\ \hline \hline
CVX $\mathds{D}\mathds{P}_n$&32.04&60.19&97.69&152.32&267.45  \\ \hline
CG $\mathds{D}\mathds{P}_n$&0.80&0.89&1.42&1.69&2.16  \\ \hline
TR $\mathds{D}\mathds{P}_n$&7.69&11.03&18.24&22.60&24.17  \\ \hline
\end{tabular}
\caption{Performance of the Doubly Stochastic Multinomial Manifold Against the Problem Dimension.}
\label{tab:2}
\end{small}
\end{table}

\begin{table}[t]
\begin{small}
\centering
\begin{tabular}{|c||c|c|c|c|c|}
\hline  
$n$&60&70&80&90&100  \\ \hline \hline
CVX $\mathds{S}\mathds{P}_n$&8.66&16.08&26.12&39.95&59.57  \\ \hline
CG $\mathds{S}\mathds{P}_n$&0.18&0.31&0.34&0.44&0.53   \\ \hline
TR $\mathds{S}\mathds{P}_n$&0.61&1.15&1.43&2.07&2.50  \\ \hline
\end{tabular}
\caption{Performance of the Symmetric Stochastic Multinomial Manifold Against the Problem Dimension.}
\label{tab:3}
\end{small}
\end{table}

\begin{table}[t]
\begin{small}
\centering
\begin{tabular}{|c||c|c|c|c|c|}
\hline  
$n$&60&70&80&90&100  \\ \hline \hline
CVX $\mathds{S}\mathds{P}^+_n$&5.43&11.39&22.84&33.77&53.79 \\ \hline
CG $\mathds{S}\mathds{P}^+_n$&0.17&0.18&0.21&0.28&0.34 \\ \hline
TR $\mathds{S}\mathds{P}^+_n$&0.58&0.78&0.93&1.32&1.90  \\ \hline
\end{tabular}
\caption{Performance of the Definite Symmetric Stochastic Multinomial Manifold Against the Problem Dimension.}
\label{tab:4}
\end{small}
\end{table}

\tref{tab:2} illustrates the performance of the proposed method in denoising a doubly stochastic matrix against the problem dimension. The table reveals a significant gain in the simulation time ranging from $39$ to $123$ fold for the first order method and from $4$ to $11$ fold for the second order algorithm as compared with the generic solver. The gain in performance can be explained by the fact that the proposed method uses the geometry of the problem efficiently unlike generic solvers which convert the problem in a standard form and solve it using standard methods. The second order method performs poorly as compared with the first order method due to the fact that the expression of the Riemannian Hessian is complex to compute.

\tref{tab:3} shows the simulation time of the symmetric doubly stochastic multinomial manifold against the problem size. One can note that the gain is more important than the one in \tref{tab:2}. Indeed, as shown in the complexity analysis section, the symmetric manifold enjoys a large dimension reduction as compared with the doubly stochastic one which makes the required ingredients easier to compute. One can note that the computation of the Riemannian Hessian of the symmetric stochastic manifold is more efficient that the doubly stochastic manifold which is reflected in a better performance against the conjugate gradient algorithm.

\tref{tab:4} displayed similar performance in the positive symmetric doubly stochastic multinomial manifold. The proposed definite multinomial manifold efficiently finds the solution. This is due to the fact that the optimal solution does not represent vanishing entries or eigenvalues as pointed out in \sref{sec:ext} which makes the retraction efficient. Such condition being not fulfilled in the upcomming couple of subsections, the performance of the positive definite manifold is omitted and a relaxed version using the symmetric manifold and regularization is presented.

The rest of the subsection confirms the linear and quadratic convergence rate behaviors of the proposed method by plotting the norm of the gradient against the iteration number of each of the manifolds and algorithms. For each of the manifolds, an optimization problem is set up using regularizers is order to reach the optimal solution to optimization problem \eref{eq:545} with $\mathcal{M}=\mathds{S}\mathds{P}_n^+$. The definition of the regularized objective functions is delayed to the next subsection for the more interesting clustering problem. Since the complexity of each step depends on the manifold, optimization problem, and used algorithm, nothing is concluded in this subsection about the efficiency of the algorithm in reaching the same solution.

\begin{figure}[t]
\centering
\includegraphics[width=1\linewidth]{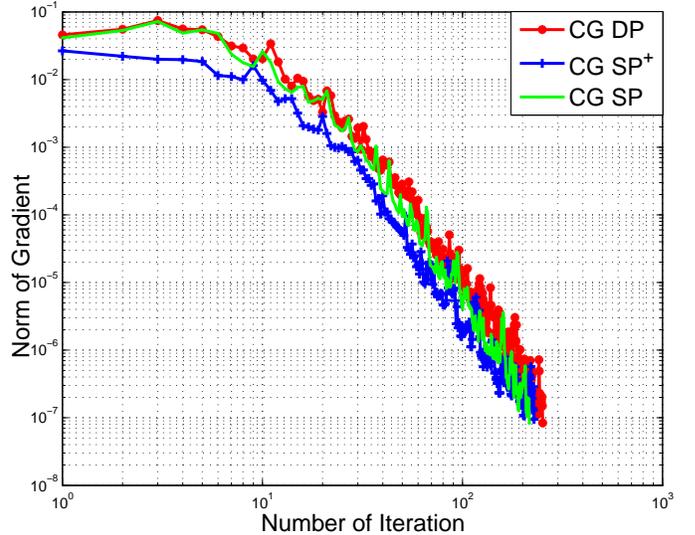}\\
\caption{Convergence rate of the conjugate gradient algorithm on the various proposed manifolds against the iteration number for a high dimension system $n=1000$.} \label{fig:4}
\end{figure}

\begin{figure}[t]
\centering
\includegraphics[width=1\linewidth]{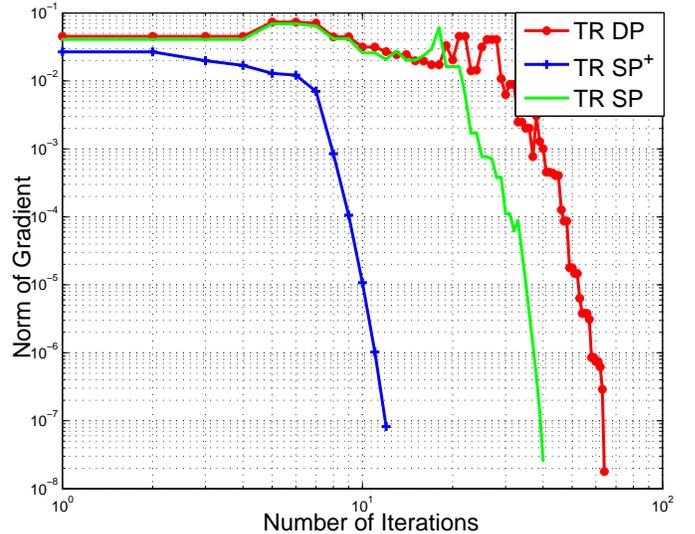}\\
\caption{Convergence rate of the trust region method on the various proposed manifolds against the iteration number for a high dimension system $n=1000$.} \label{fig:5}
\end{figure}

\fref{fig:4} plots the convergence rate of the first order method using the doubly stochastic, symmetric, and positive manifolds. The figures clearly shows that the conjugate gradient algorithm exhibits a linear convergence rate behavior similar to the one of unconstrained optimization. This is mainly due to the fact that the Riemannian optimization approach convert a constrained optimization into an unconstrained one over a constrained set. 

\fref{fig:5} shows that the trust region method has a super linear, i.e., quadratic, convergence behavior with respect to the iteration number. The figure particularly show that the quadratic rate is achieved after a number of iterations which can be explained by the fact that our implementation uses a general retraction instead of the optimal (and complex) Exponential map.

\subsection{Similarity Clustering via Convex Programming}

This section suggests using the proposed framework to solve the convex clustering problem \cite{7541267}. Given an entry-wise non-negative similarity matrix $\mathbf{A}$ between $n$ data points, the goal is to cluster these data points into $r$ clusters. Similar to \cite{7541267}, the matrix is generated by assuming a noisy block stochastic model with $3$ blocks, a connection probability of $0.7$ intra-cluster and $0.2$ inter-clusters, and a noise variance of $0.2$. Under the above conditions, the reference guarantees the recovery of the clusters by solving the following optimization problem:
\begin{align}
\min_{\substack{\mathbf{X} \in \mathds{S}\mathds{P}_n \\ \mathbf{X}\succeq \mathbf{0} }} ||\mathbf{A}-\mathbf{X}||_F^2 + \lambda\text{Tr}(\mathbf{X}),\label{eq:456}
\end{align}
wherein $\lambda$ is the regulizer parameter whose expression is derived in \cite{7541267}. The optimal solution to the above problem is a block matrix (up to a permutation) of rank equal to the number of clusters. Due to such rank deficiency of the optimal solution, the definite positive manifold cannot be used to solve the above problem. Therefore, we reformulate the problem on $\mathds{S}\mathds{P}_n$ by adding the adequate regulizers as below:
\begin{align}
\min_{\mathbf{X} \in \mathds{S}\mathds{P}_n} ||\mathbf{A}-\mathbf{X}||_F^2 + \lambda\text{Tr}(\mathbf{X}) + \rho(||\mathbf{X}||_* - \text{Tr}(\mathbf{X})),\label{eq:125} 
\end{align} 
wherein $\rho$ is the regularization parameter. The expression of such regulizer can be obtained by expressing the Lagrangian of the original problem and deriving the expression of the Lagrangian multiplier. However, this falls outside the scope of this paper. Clearly, the expression $||\mathbf{X}||_* - \text{Tr}(\mathbf{X}) = \sum_{i=1}^n|\lambda_i|-\lambda_i$ is positive and equal to zero if and only if all the eigenvalues are positive which concludes that $\mathbf{X}$ is positive. Similarly, the problem can be reformulated on $\mathds{D}\mathds{P}_n$ as follows:
\begin{align}
\min_{\mathbf{X} \in \mathds{S}\mathds{P}_n} f(\mathbf{X}) + \rho(||\mathbf{X}||_* - \text{Tr}(\mathbf{X})) + \mu(||\mathbf{X}-\mathbf{X}^T||_F^2), \label{eq:126} 
\end{align}
where $f(\mathbf{X})$ is the original objective function in \eref{eq:456} regularized with $\rho$ and $\mu$ to promote positiveness and symmetry.

\begin{figure}[t]
\centering
\includegraphics[width=1\linewidth]{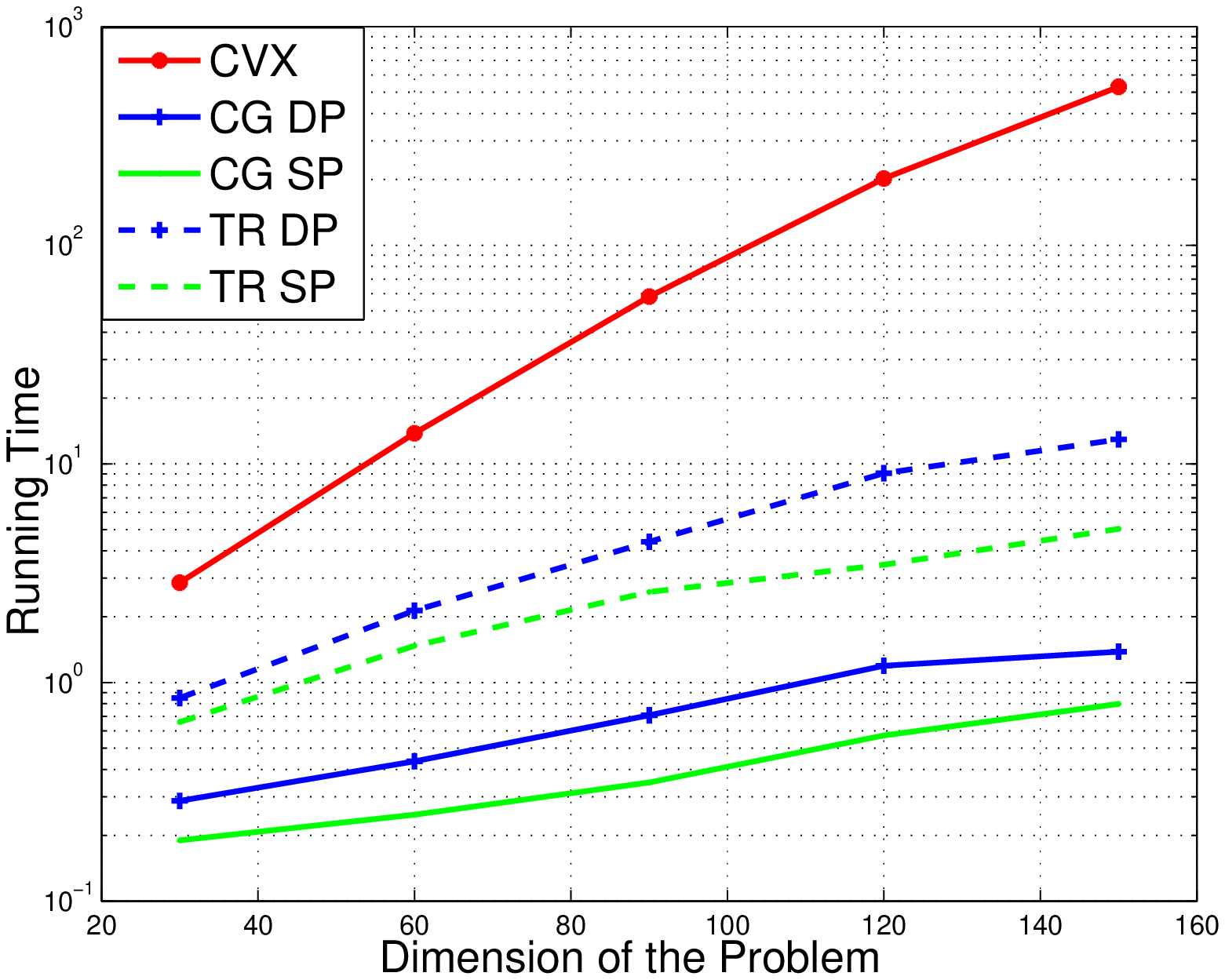}\\
\begin{small}
\begin{tabular}{|c||ccccc|}
\hline 
$n$&30&60&90&120&150  \\ \hline\hline
CVX $\mathds{S}\mathds{P}^+_n$&90.53&383.69&881.49&1590.41&2499.86  \\ \hline
CG $\mathds{S}\mathds{P}_n$&90.53&383.67&881.44&1590.33&2499.74  \\ \hline
TR $\mathds{S}\mathds{P}_n$&90.77&384.02&881.98&1590.90&2500.36  \\ \hline
CG $\mathds{D}\mathds{P}_n$&90.53&383.66&881.46&1590.32&2499.76  \\ \hline
TR $\mathds{D}\mathds{P}_n$&90.87&384.15&881.92&1590.58&2500.06  \\ \hline
\end{tabular}
\end{small}
\caption{Performance of the doubly stochastic and symmetric multinomial manifolds in solving the convex clustering problem in terms of running time and objective cost.} \label{fig:6}
\end{figure}

\fref{fig:6} plots the running time required to solve the convex clustering problem and show the achieved original cost \eref{eq:456} for the different optimization methods. Clearly, the proposed Riemannian optimization algorithms largely outperform the standard approach with gains ranging from $15$ to $665$ fold for the first order methods. The precision of the algorithms is satisfactory as they achieve the CVX precision in almost all experiments. Also note that using the symmetric multinomial manifold produces better results. This can be explained by the fact that not only the objective function \eref{eq:125} is simpler than \eref{eq:126} but also by the fact that the manifold contains less degrees of freedom which makes the projections more efficient. 

\subsection{Clustering by Low-Rank Doubly Stochastic Matrix Decomposition}

This last part of the simulations tests the performance of the proposed method for clustering by low-rank doubly stochastic matrix decomposition in the setting proposed in \cite{Yang:2012:CLD:3042573.3042666}. Given a similarity matrix as in the previous section, the authors in the above reference claim that a suitable objective function to determine the clusters structure is the following non-convex cost:
\begin{align}
\min_{\substack{\mathbf{X} \in \mathds{S}\mathds{P}_n \\ \mathbf{X}\succeq \mathbf{0} }} -\sum_{i,j} \mathbf{A}_{ij} \log\left(\sum_k\cfrac{\mathbf{X}_{ik}\mathbf{X}_{jk}}{\sum_v \mathbf{X}_{vk}}\right) - (\alpha-1)\sum_{ij}\log\left(\mathbf{X}_{ij}\right) \nonumber
\end{align}

The authors propose a specialized algorithm, known as ``Relaxed MM", to efficiently solve the problem above. This section suggests solving the above problem using the positive and the symmetric multinomial manifold (with the proper regularization as shown in the previous subsection). In order to reach the same solution, all algorithms are initialized with the same value. The objective function achieved by the algorithm of \cite{Yang:2012:CLD:3042573.3042666} is taken as a reference, and the other algorithms stop as soon as their cost drops below such value.

\begin{figure}[t]
\centering
\includegraphics[width=1\linewidth]{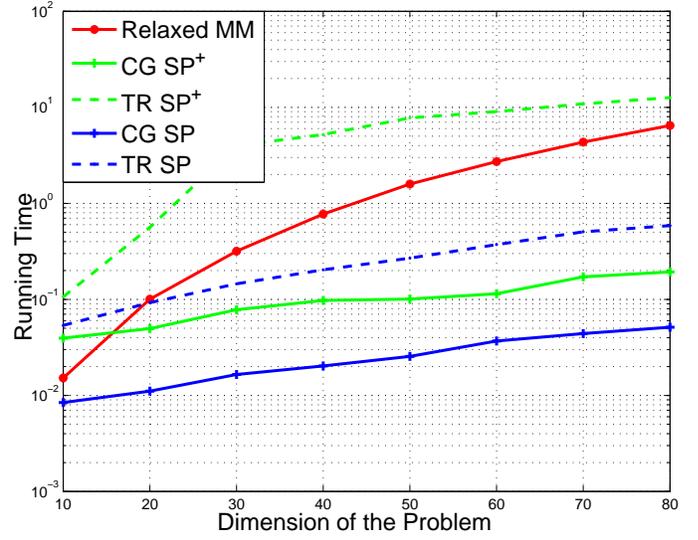}\\
\caption{Performance of the positive and symmetric multinomial manifolds in solving the non-convex clustering problem against the Relaxed MM algorithm.} \label{fig:7}
\end{figure}

\fref{fig:7} illustrates the running time of the different algorithms in order to reach the same solution. The plot reveals that the proposed framework is highly efficient in high dimension with significant gain over the specialized algorithm. The performance of the first order method is noticeably better than the second order one. This can be explained by the complexity of deriving the Riemannian Hessian. In practical implementations, one would use an approximation of the Hessian in a similar manner as the quasi-Newton methods, e.g., BHHH, BFGS. Finally, one can note that the symmetric multinomial performs better than the positive one which can be explained by the fact that the optimal solution has vanishing eigenvalues which make the retraction on the cone of positive matrices non-efficient.

\section{Conclusion} \label{sec:con}

This paper proposes a Riemannian geometry-based framework to solve a subset of convex (and non-convex) optimization problems in which the variable of interest represents a multidimensional probability distribution function. The optimization problems are reformulated from constrained optimizations into unconstrained ones over a restricted search space. The fundamental philosophy of the Riemannian optimization is to take advantage of the low-dimension manifold in which the solution lives and to use efficient unconstrained optimization algorithms while ensuring that each update remains feasible. The geometrical structure of the doubly stochastic, symmetric stochastic, and the definite multinomial manifold is studied and efficient first and second order optimization algorithms are proposed. Simulation results reveal that the proposed approach outperforms conventional generic and specialized solvers in efficiently solving the problem in high dimensions.

\appendices
\numberwithin{equation}{section}

\section{Computation of the Tangent Space} \label{app1}

This section computes the tangent space of the different manifold of interest in this paper. Recall that the tangent space $\mathcal{T}_x\mathcal{M}$ of the manifold $\mathcal{M}$ at $x$ is the $d$-dimensional vector space generated by the derivative of all curves going through $x$. Therefore, the computation of such tangent space requires considering a generic smooth curve $\gamma(t): \mathds{R} \longrightarrow \mathcal{M}$ such as $\gamma(0)=x$ and $\gamma(t) \in \mathcal{M}$ for some $t$ in the neighborhood of $0$. The evaluation of the derivative of such parametric curves at the origin generates a vector space $\mathcal{D}$ such that $\mathcal{T}_x\mathcal{M} \subseteq \mathcal{D}$. Two approaches can be used to show the converse. If the dimension of the manifold is known apriori and match the dimension of $\mathcal{D}$, then $\mathcal{T}_x\mathcal{M} = \mathcal{D}$. Such approach is referred as dimension count. The second and more direct method is to consider each element $d \in \mathcal{D}$ and construct a smooth curve $\gamma(t) \in \mathcal{M}$ for some $t \in \mathcal{I} \subset \mathds{R}$ such that $\gamma(0)=x$ and $\gamma^\prime(0)=d$. For illustration purposes, the first and second subsections use the former and latter approaches, respectively.  

\subsection{Proof of \preref{pre1}}

Recall the definition of the doubly stochastic manifold $\mathds{D}\mathds{P}_{n} = \left\{ \mathbf{X} \in \mathds{R}^{n \times n} \big|  X_{ij} > 0,\ \mathbf{X}\mathbf{1}=\mathbf{1},\ \mathbf{X}^T\mathbf{1}=\mathbf{1} \right\}$. Consider an $\mathbf{X} \in \mathds{D}\mathds{P}_{n}$ and let $\mathbf{X}(t)$ be a smooth curve such that $\mathbf{X}(0)=\mathbf{X}$. Since $\mathbf{X}(t) \in \mathds{D}\mathds{P}_{n}$ for some $t$ in the neighborhood of the origin, then the curve satisfies:
\begin{align}
\mathbf{X}(t)\mathbf{1}=\mathbf{1} &\Rightarrow \dot{\mathbf{X}}(t)\mathbf{1}=\mathbf{0}\\
(\mathbf{X}(t))^T\mathbf{1}=\mathbf{1} &\Rightarrow (\dot{\mathbf{X}}(t))^T\mathbf{1}=\mathbf{0}
\end{align}
Differentiating both equations above concludes that the tangent space is a subset of 
\begin{align}
\mathcal{T}_{\mathbf{X}}\mathds{D}\mathds{P}_{n} \subseteq \left\{\mathbf{Z} \in \mathds{R}^{n \times n} \big| \mathbf{Z}\mathbf{1}=\mathbf{0},\ \mathbf{Z}^T\mathbf{1}=\mathbf{0} \right\}.
\end{align}

From the Birkhoff-von Neumann theorem, the degrees of freedom of doubly stochastic matrices is $(n-1)^2$. Similarly, one can note that the above space is generated by $2n-1$ independent linear equations (the sum of the last column can be written as the difference of the sum of all row and the some of all except the last column). In other words, the dimension of the space is $n^2-(2n-1)=(n-1)^2$. Therefore, a dimension count argument concludes that the tangent space has the above expression.

\subsection{Proof of \preref{pre2}}

The symmetric multinomial manifold has the following expression $\mathds{S}\mathds{P}_{n} = \left\{ \mathbf{X} \in \mathds{R}^{n \times n} \big|  X_{ij} > 0,\ \mathbf{X}\mathbf{1}=\mathbf{1},\ \mathbf{X} = \mathbf{X}^T  \right\}$. Therefore, a smooth curve $\mathbf{X}(t)$ that goes through a point $\mathbf{X} \in \mathds{S}\mathds{P}_{n}$ satisfies:
\begin{align}
\mathbf{X}(t)=(\mathbf{X}(t))^T &\Rightarrow \dot{\mathbf{X}}(t)=(\dot{\mathbf{X}}(t))^T\\
\mathbf{X}(t)\mathbf{1}=\mathbf{1} &\Rightarrow \dot{\mathbf{X}}(t)\mathbf{1}=\mathbf{0}
\end{align}
which concludes that the tangent space $\mathcal{T}_{\mathbf{X}}\mathds{D}\mathds{P}_{n}$ is included in the set $\left\{\mathbf{Z} \in \mathcal{S}_n \big| \mathbf{Z}\mathbf{1}=\mathbf{0}\right\}$. Now consider $\mathbf{Z}$ in the above set and the smooth curve $\gamma(t) = \mathbf{X} + t\mathbf{Z}$. Clearly, $\gamma(t)=(\gamma(t))^T$ for all $t\in \mathds{R}$. Furthermore, we have:
\begin{align}
\gamma(t)\mathbf{1}=\mathbf{X}\mathbf{1} + t\mathbf{Z}\mathbf{1}=\mathbf{X}\mathbf{1}=\mathbf{1}
\end{align}
Finally, since $X_{ij}>0$ defines an open set, then there exists an interval $\mathcal{I} \subset \mathds{R}$ such that $\gamma(t)_{ij} >0$. Finally, it is clear that $\gamma(\mathbf{0})=\mathbf{X}$ and $\gamma^\prime(\mathbf{0})=\mathbf{Z}$ which concludes that:
\begin{align}
\mathcal{T}_{\mathbf{X}}\mathds{S}\mathds{P}_{n} = \left\{\mathbf{Z} \in \mathcal{S}_n \big| \mathbf{Z}\mathbf{1}=\mathbf{0}\right\}.
\end{align}

\section{Orthogonal Projection on the Tangent Space} \label{app2}

This section describes the general procedure to obtain the orthogonal projection to manifolds of interest in this paper. For a manifold $\mathcal{M}$ embedded in a vector space $\mathcal{E}$, the orthogonal projection $\Pi_x(z)$ projects a point $z \in \mathcal{E}$ onto the tangent space $\mathcal{T}_x\mathcal{M}$ for some $x \in \mathcal{M}$. The term orthogonal herein refers to the fact that the difference $z - \Pi_x(z)$ is orthogonal to the space $\mathcal{T}_x\mathcal{M}$ for the inner product $\langle.,.\rangle_x$. Therefore, the first step in deriving the expression of the orthogonal projection onto the tangent space $\mathcal{T}_x\mathcal{M}$ is to determine its orthogonal complement $\mathcal{T}^\perp_x\mathcal{M}$ defined as:
\begin{align}
\mathcal{T}^\perp_x\mathcal{M} = \{ z^\perp \in \mathcal{E} \  | \  \langle z^\perp,z\rangle_x, \ \forall \ z \in \mathcal{T}_x\mathcal{M} \}
\end{align}

As the dimension of $\mathcal{T}^\perp_x\mathcal{M}$ can be written as $\text{dim}(\mathcal{E})-\text{dim}(\mathcal{T}_x\mathcal{M})$, a typically method for deriving the expression of the orthogonal complement is to find a generating family for the space and check its dimension. Now, let $\Pi^\perp_x(z)$ be the orthogonal projection onto the orthogonal complement $\mathcal{T}^\perp_x\mathcal{M}$, each point in the ambient space can be decomposed as:
\begin{align}
z = \Pi_x(z)+\Pi^\perp_x(z), \ \forall \ z \in \mathcal{E}, \ \forall \ x \in \mathcal{M}.
\end{align}

Using the expressions of both the tangent set and its complement, the above equation allows deriving the expressions of both projections simultaneously. The next subsections compute the orthogonal projection on the set of double stochastic and symmetric multinomial manifolds using the described method.

\subsection{Proof of \thref{th1}}

As stated earlier, the first step in deriving the expression of the orthogonal projection onto the tangent space, one needs to derive the expression on the orthogonal complement which is given in the following lemma.
\begin{lemma}
The orthogonal complement of the tangent space of the doubly stochastic multinomial has the following expression:
\begin{align}
\mathcal{T}^\perp_{\mathbf{X}}\mathds{D}\mathds{P}_{n} = \left\{ \mathbf{Z}^\perp \in \mathds{R}^{n \times n} \big| \mathbf{Z}^\perp = (\alpha \mathbf{1}^T + \mathbf{1} \beta^T ) \odot \mathbf{X} \right\}
\end{align}
for some $\alpha,\beta \in \mathds{R}^{n}$.
\label{lem2}
\end{lemma}

\begin{proof}
As stated in the introduction of the section, the computation of the orthogonal complement of the tangent space requires on deriving a basis for the space and counting the dimension. Let $\mathbf{Z}^\perp \in \mathcal{T}^\perp_{\mathbf{X}}\mathds{D}\mathds{P}_{n}$ and $\mathbf{Z} \in \mathcal{T}_{\mathbf{X}}\mathds{D}\mathds{P}_{n}$, the inner product can be written as:
\begin{align}
\langle \mathbf{Z}^\perp,\mathbf{Z} \rangle_{\mathbf{X}} &= \text{Tr}((\mathbf{Z}^\perp \oslash \mathbf{X})\mathbf{Z}^T) \nonumber \\
&= \text{Tr}((\alpha \mathbf{1}^T + \mathbf{1} \beta^T)\mathbf{Z}^T) \nonumber \\
&= \alpha^T \mathbf{Z} \mathbf{1}  + \beta^T \mathbf{Z}^T\mathbf{1} 
\end{align}
But $\mathbf{Z} \mathbf{1} = \mathbf{Z}^T\mathbf{1} = \mathbf{0}$ by definition of the tangent space. Therefore, we have  $\langle \mathbf{Z},\mathbf{Z}^\perp \rangle_x, \ \forall \ \mathbf{Z} \in \mathcal{T}_{\mathbf{X}}\mathds{D}\mathds{P}_{n}$. 

Finally, one can note that the dimension of set is $2n-1$ which is the correct dimension for $\mathcal{T}^\perp_{\mathbf{X}}\mathds{D}\mathds{P}_{n}$. Therefore, the derived set is the orthogonal complement of the tangent space.
\end{proof}

Let $\mathbf{Z} \in \mathds{R}^{n \times n}$ be a vector in the ambient space and $\mathbf{X} \in \mathds{D}\mathds{P}_{n}$. The expression of the projections are obtained using the following decomposition:
\begin{align}
\mathbf{Z} &= \Pi_{\mathbf{X}}(\mathbf{Z}) + \Pi^\perp_{\mathbf{X}}(\mathbf{Z}) \nonumber \\
\mathbf{Z}\mathbf{1}  &= \Pi_{\mathbf{X}}(\mathbf{Z})\mathbf{1} +\Pi^\perp_{\mathbf{X}}(\mathbf{Z}) \mathbf{1} \label{eq:app22}
\end{align}
However, by definition of the tangent space, the first term in the right hand side in the above equation vanishes. Similarly, from \lref{lem2}, the second term can be replaced by its $(\alpha \mathbf{1}^T + \mathbf{1} \beta^T ) \odot \mathbf{X}$. Therefore the first equation implies 
\begin{align}
\mathbf{Z}\mathbf{1} &= ((\alpha \mathbf{1}^T + \mathbf{1} \beta^T ) \odot \mathbf{X})\mathbf{1} \nonumber \\
\sum_{j=1}^n \mathbf{Z}_{ij} &= \sum_{j=1}^n (\alpha_i + \beta_j)\mathbf{X}_{ij}, 1 \leq i \leq n \nonumber \\
\sum_{j=1}^n \mathbf{Z}_{ij} &= \alpha_i + \sum_{j=1}^n \beta_j\mathbf{X}_{ij}, 1 \leq i \leq n \nonumber \\
\mathbf{Z}\mathbf{1} &= \alpha + \mathbf{X}\beta \label{eq:app25}
\end{align}
A similar argument allows to conclude that $\mathbf{Z}^T\mathbf{1} = \mathbf{X}^T\alpha+ \beta$. Grouping the equations above gives the following system of equations:
\begin{align}
\begin{pmatrix}
\mathbf{Z}\mathbf{1} \\
\mathbf{Z}^T\mathbf{1}
\end{pmatrix} = 
\begin{pmatrix}
\mathbf{I} & \mathbf{X} \\
\mathbf{X}^T & \mathbf{I}
\end{pmatrix}
\begin{pmatrix}
\alpha \\
\beta
\end{pmatrix} \label{eq:24}
\end{align}

Even though the matrix $\begin{pmatrix}
\mathbf{I} & \mathbf{X} \\
\mathbf{X}^T & \mathbf{I}
\end{pmatrix}$ is rank deficient to the present of the null vector at $\begin{pmatrix}
\mathbf{1}  \\
-\mathbf{1}
\end{pmatrix}$, the systems admits infinitely many solutions. Indeed, from the orthogonal complement identification of the range and null space of a matrix $\mathbf{A}$, i.e., $\mathcal{R}(\mathbf{A}) = \mathcal{N}^\perp(\mathbf{A})$, it is sufficient to show that the vector of interest is orthogonal to the null space of the matrix of interest as follows:
\begin{align}
\begin{pmatrix}
\mathbf{Z}\mathbf{1} \\
\mathbf{Z}^T\mathbf{1}
\end{pmatrix}^T \begin{pmatrix}
\mathbf{1}  \\
-\mathbf{1}
\end{pmatrix} &= \mathbf{1}^T\mathbf{Z}^T\mathbf{1} - \mathbf{1}^T\mathbf{Z}\mathbf{1} \nonumber \\
&= \mathbf{1}^T\mathbf{Z}\mathbf{1}-\mathbf{1}^T\mathbf{Z}\mathbf{1} = \mathbf{0}
\end{align}

A particular solution to the system is the solve for $\beta$ as a function of $\alpha$ and solve for $\alpha$ which gives the following solution
\begin{align}
\alpha &= (\mathbf{I} - \mathbf{X}\mathbf{X}^T)^{\dagger}(\mathbf{Z} - \mathbf{X}\mathbf{Z}^T)\mathbf{1}  \\
\beta &= \mathbf{Z}^T\mathbf{1} - \mathbf{X}^{T}\alpha \label{eq:app23}
\end{align}
Finally, rearranging the terms in \eref{eq:app22} allows to conclude that the orthogonal projection onto the tangent space has the following expression
\begin{align}
\Pi_\mathbf{X}(\mathbf{Z}) = \mathbf{Z} - (\alpha \mathbf{1}^T + \mathbf{1} \beta^T ) \odot \mathbf{X},
\end{align}
wherein $\alpha$ and $\beta$ are obtained according to \eref{eq:app23} or more generally \eref{eq:24}.

\subsection{Proof of \thref{th2}}

The proof of this theorem is closely related to the proof of \thref{th1}. Indeed, the next section shows that the orthogonal projection on the symmetric double stochastic multinomial manifold is a special case of the projection on the doubly stochastic multinomial manifold. This section provides a more direct proof that does not reply on the previously derived result. The expression of the orthogonal complement of the tangent space is given in the following lemma.
\begin{lemma}
The orthogonal complement of the tangent space of the symmetric multinomial can be represented by the following set:
\begin{align}
\mathcal{T}^\perp_{\mathbf{X}}\mathds{S}\mathds{P}_{n} = \left\{ \mathbf{Z}^\perp \in \mathds{S}_{n} \big| \mathbf{Z}^\perp = (\alpha \mathbf{1}^T + \mathbf{1} \alpha^T ) \odot \mathbf{X} \right\}
\end{align}
for some $\alpha \in \mathds{R}^{n}$.
\label{lem3}
\end{lemma}

\begin{proof}
The proof of this lemma uses similar steps like the one of \lref{lem2} and thus is omitted herein.
\end{proof}

Let $\mathbf{Z} \in \mathds{R}^{n \times n}$ be a vector in the ambient space and $\mathbf{X} \in \mathds{D}\mathds{P}_{n}$. The decomposition of $\mathbf{Z}$ gives the following:
\begin{align}
\mathbf{Z} &= \Pi_{\mathbf{X}}(\mathbf{Z}) + \Pi^\perp_{\mathbf{X}}(\mathbf{Z}) \nonumber \\
\mathbf{Z}\mathbf{1}  &= \Pi_{\mathbf{X}}(\mathbf{Z})\mathbf{1} +\Pi^\perp_{\mathbf{X}}(\mathbf{Z}) \mathbf{1}  \nonumber \\
\mathbf{Z}\mathbf{1}  &= ((\alpha \mathbf{1}^T + \mathbf{1} \alpha^T ) \odot \mathbf{X}) \mathbf{1} \nonumber \\
\mathbf{Z}\mathbf{1} &=  \alpha + \mathbf{X}\alpha = (\mathbf{I}+\mathbf{X})\alpha 
\nonumber \\
\alpha &=  (\mathbf{I}+\mathbf{X})^{-1}\mathbf{Z}\mathbf{1}, \label{eq:app26}
\end{align}
wherein the steps of the computation are obtained in a similar fashion as the one in \eref{eq:app25}. Therefore, the orthogonal projection on the tangent set of symmetric doubly stochastic multinational manifold is given by:
\begin{align}
\Pi_\mathbf{X}(\mathbf{Z}) = \mathbf{Z} - (\alpha \mathbf{1}^T + \mathbf{1} \alpha^T ) \odot \mathbf{X},
\end{align}
with $\alpha$ being derived in \eref{eq:app26}.

\subsection{Relationship Between the Orthogonal Projections}

One can note that the projection onto the tangent space of the symmetric multinomial is a special case of the projection onto the tangent space of the doubly stochastic manifold when both the point on the manifold $\mathbf{X}$ and the ambient vector $\mathbf{Z}$ are symmetric. In other words, the projection can be written as
\begin{align}
\Pi_\mathbf{X}(\mathbf{Z}) = \mathbf{Z} - (\alpha \mathbf{1}^T + \mathbf{1} \beta^T ) \odot \mathbf{X},
\end{align}
with
\begin{align}
\alpha &= (\mathbf{I} - \mathbf{X}\mathbf{X}^T)^{\dagger}(\mathbf{Z} - \mathbf{X}\mathbf{Z}^T)\mathbf{1}  \\
\beta &= (\mathbf{Z}^T - (\mathbf{X}^{-T} - \mathbf{X})^{\dagger}(\mathbf{Z} - \mathbf{X}\mathbf{Z}^T))\mathbf{1}
\end{align}
and the additional identities $\mathbf{X}=\mathbf{X}^T$ and $\mathbf{Z}=\mathbf{Z}^T$. Using the economic eigenvalue decomposition of the symmetric matrix $\mathbf{X}=\mathbf{U}\mathbf{\Lambda}\mathbf{U}^T$, vector $\alpha$ can be expressed as:
\begin{align}
\alpha &= (\mathbf{I} - \mathbf{X}\mathbf{X}^T)^{\dagger}(\mathbf{Z} - \mathbf{X}\mathbf{Z}^T)\mathbf{1} \nonumber \\
&= (\mathbf{I} - \mathbf{X}\mathbf{X}^T)^{\dagger}\mathbf{Z}\mathbf{1} -  (\mathbf{I} - \mathbf{X}\mathbf{X}^T)^{\dagger}\mathbf{X}\mathbf{Z}^T\mathbf{1} \nonumber \\
&= (\mathbf{I} - \mathbf{X}^2)^{\dagger}\mathbf{Z}\mathbf{1} -  (\mathbf{X}^{-1} - \mathbf{X})^{\dagger}\mathbf{Z}\mathbf{1} \nonumber \\
&= 
\mathbf{U} \left[(\mathbf{I} - \mathbf{\Lambda}^2)^{-1}-  (\mathbf{\Lambda}^{-1} - \mathbf{\Lambda})^{-1} \right]\mathbf{U}^T \mathbf{Z}\mathbf{1}
\end{align}
The inner matrix is a diagonal one with diagonal entries equal to:
\begin{align}
\cfrac{1}{1-\lambda^2} - \cfrac{1}{\cfrac{1}{\lambda}-\lambda} &= \cfrac{1}{1-\lambda^2} - \cfrac{\lambda}{1-\lambda^2} = \cfrac{1}{1+\lambda}
\end{align}
Therefore, $(\mathbf{I} - \mathbf{\Lambda}^2)^{-1}-  (\mathbf{\Lambda}^{-1} - \mathbf{\Lambda})^{-1} = (\mathbf{I} + \mathbf{\Lambda})^{-1}$ which gives the final expression of $\alpha$ as:
\begin{align}
\alpha &= (\mathbf{I} + \mathbf{X})^{-1}\mathbf{Z}\mathbf{1}.
\end{align}
Finally, the expression of $\beta$ is given by:
\begin{align}
\beta &= \mathbf{Z}^T\mathbf{1}-\mathbf{X}^T\alpha = (\mathbf{I}-\mathbf{X}(\mathbf{I} + \mathbf{X})^{-1})\mathbf{Z}\mathbf{1} \nonumber \\
&= \mathbf{U}\left[ \mathbf{I}-\mathbf{\Lambda}(\mathbf{I} + \mathbf{\Lambda})^{-1}) \right]\mathbf{U}^T\mathbf{Z}\mathbf{1},
\end{align} 
with the inner matrix equals to:
\begin{align}
1 - \lambda \cfrac{1}{1+\lambda} = \cfrac{1}{1+\lambda} \Rightarrow (\mathbf{I}-\mathbf{X}(\mathbf{I} + \mathbf{X})^{-1}) = (\mathbf{I} + \mathbf{X})^{-1} \nonumber
\end{align}
Therefore, we conclude that $\beta = (\mathbf{I} + \mathbf{X})^{-1}\mathbf{Z}\mathbf{1} = \alpha$ which is in accordance with the result derived in \thref{th2}.
\begin{remark}
Note that the link between both expressions can be obtained easier by assuming that $(\mathbf{I} - \mathbf{X}\mathbf{X}^T)$ is invertible. Indeed, for example the expression of $\alpha$ can be easily computed as:
\begin{align}
\alpha &= (\mathbf{I} - \mathbf{X}\mathbf{X}^T)^{-1}(\mathbf{Z} - \mathbf{X}\mathbf{Z}^T)\mathbf{1} \nonumber \\
 &= (\mathbf{I} - \mathbf{X}\mathbf{X})^{-1}(\mathbf{Z} - \mathbf{X}\mathbf{Z})\mathbf{1} \nonumber \\
&= (\mathbf{I} + \mathbf{X})^{-1}(\mathbf{I} -\mathbf{X})^{-1} ( \mathbf{I} - \mathbf{X})\mathbf{Z}\mathbf{1} \nonumber \\
&= (\mathbf{I} + \mathbf{X})^{-1}\mathbf{Z}\mathbf{1}
\end{align}
However, due to eigenvalue at $1$ from $\mathbf{X}\mathbf{1}=\mathbf{1}$, such proof is not valid and we need to use the pseudo-inverse as shown in the section above.
\end{remark}

\section{Retraction on Embedded Manifolds} \label{app3}

This section exploits the vector space structure of the embedding space to design efficient, i.e., low-complexity, retractions on the manifolds of interest in this paper. The construction of the retraction rely on the following theorem whose proof can be found in \cite{AbsMahSep2008}.
\begin{theorem}
Let $\mathcal{M}$ be an embedded manifold of the Euclidean space $\mathcal{E}$ and let $\mathcal{N}$ be an abstract manifold such that dim($\mathcal{M}$) + dim($\mathcal{N}$) = dim($\mathcal{E}$). Assume that there is a diffeomorphism
\begin{align}
\phi: \mathcal{M} \times \mathcal{N} &\longrightarrow \mathcal{E}^* \nonumber \\
 (\mathbf{F},\mathbf{G}) &\longmapsto \phi(\mathbf{F},\mathbf{G})
\end{align}
where $\mathcal{E}^*$ is an open subset of $\mathcal{E}$, with a neutral element $\mathbf{I} \in \mathbf{N}$ satisfying
\begin{align}
\phi(\mathbf{F},\mathbf{I}) = \mathbf{F}, \ \forall \ \mathbf{F} \in \mathcal{M}
\end{align}
Under the above assumption, the mapping 
\begin{align}
R_x: \mathcal{T}_x\mathcal{M} &\longrightarrow \mathcal{M}  \nonumber \\
\xi_x  &\longmapsto  R_x(\xi_x) = \pi_1(\phi^{-1}(x+\xi_x)),
\end{align}
where $\pi_1: \mathcal{M} \times \mathcal{N} \longrightarrow \mathcal{M}: (\mathbf{F},\mathbf{G}) \longmapsto \mathbf{F}$ is the projection onto the first component, defines a retraction on the manifold $\mathcal{M}$ for all $x \in \mathcal{M}$ and $\xi_x$ in the neighborhood of $0_x$.
\label{th3}
\end{theorem}

The upcoming sections take advantage of the matrix decomposition to design a mapping $\phi$. Interestingly, the inverse of the map $\phi$ turns out to be straightforward to compute even though the projection on the doubly stochastic matrices space is challenging. 

\subsection{Proof of \thref{th4}}

This subsection uses the Sinkhorn's theorem \cite{8541523974} to derive an expression for the mapping $\phi$. The Sinkhorn's theorem states:
\begin{theorem}
Let $\mathbf{A} \in \overline{\mathds{R}}^{n \times n}$ be an element-wise positive matrix. There exists two strictly positive diagonal matrices $\mathbf{D}_1$ and $\mathbf{D}_2$ such that $\mathbf{D}_1\mathbf{A} \mathbf{D}_2$ is doubly stochastic.
\end{theorem}

Due the invariance of the above theorem for scaling $\mathbf{D}_1$ and $\mathbf{D}_2$, the rest of the paper assumes that $(\mathbf{D}_1)_{11}=1$ without loss of generality. Define the $\phi$ mapping as follows:
\begin{align}
\phi: \mathds{D}\mathds{P}_n \times \overline{\mathds{R}}^{2n-1} &\longrightarrow \overline{\mathds{R}}^{n \times n} \nonumber \\
 \left(\mathbf{A},\begin{pmatrix}
 d_1 \\ d_2
 \end{pmatrix}\right) &\longmapsto \text{diag}(1,d_1) \mathbf{A} \text{diag}(d_2).
\end{align}

Note that $\overline{\mathds{R}}^{2n-1}$ is an open subset of $\mathds{R}^{2n-1}$ and thus is a manifold by definition. Similarly, $\overline{\mathds{R}}^{n \times n}$ is an open subset of $\mathds{R}^{n \times n}$. Finally, $\text{dim}(\mathds{D}\mathds{P}_n) + \text{dim}( \overline{\mathds{R}}^{2n-1}) = (n-1)^2+2n-1 = n^2 = \text{dim}(\mathds{R}^{n \times n})$. Also, the all one element of $\mathds{R}^{2n-1}$ satisfies $\phi(\mathbf{A},\mathbf{1}) = \mathbf{A}$.

Clearly, the mapping $\phi$ is smooth by the smoothness of the matrix product. The existence of the inverse map is guaranteed by the Sinkhorn's theorem. Such inverse map is obtained through the Sinkhorn's algorithm \cite{8541523974} that scales the rows and columns of the matrix, i.e., the inverse map is smooth. Finally, we conclude that $\phi$ represents a diffeomorphism.  

Using the result of \thref{th3}, we conclude that $\pi_1(\phi^{-1}(\mathbf{X}+\xi_{\mathbf{X}}))$ is a valid retraction for $\xi_{\mathbf{X}}$ in the neighborhood of $\mathbf{0}_{\mathbf{X}}$, i.e., $(\mathbf{X}+\xi_{\mathbf{X}}) \in \overline{\mathds{R}}^{n \times n} $ which can explicitly written as $\mathbf{X}_{ij} > - \left(\xi_{\mathbf{X}}\right)_{ij},\ 1 \leq i,j \leq n$. Using the property of the manifold and its tangent space, the inverse map reduce the identity. Indeed, it holds true that:
\begin{align}
(\mathbf{X}+\xi_{\mathbf{X}})\mathbf{1} &= \mathbf{X}\mathbf{1}+\xi_{\mathbf{X}}\mathbf{1} = \mathbf{1}+ \mathbf{0} =\mathbf{1} \\
(\mathbf{X}+\xi_{\mathbf{X}})^T\mathbf{1} &= \mathbf{X}^T\mathbf{1}+\xi_{\mathbf{X}}^T\mathbf{1} = \mathbf{1}+ \mathbf{0} =\mathbf{1}
\end{align}
Therefore, the canonical retraction is defined by  $R_{\mathbf{X}}(\xi_{\mathbf{X}}) = \mathbf{X} + \xi_{\mathbf{X}}$.
 
\subsection{Proof of \corref{cor1}}

The proof of this corollary follows similar steps than the one of \thref{th4}. However, instead of using the Sinkhorn's theorem to find the adequate matrix decomposition, we use its extension to the symmetric case known as the DAD theorem \cite{CSIMA1972147} given below:
\begin{theorem}
Let $\mathbf{A} \in \overline{\mathcal{S}}_n$ be a symmetric, element-wise positive matrix. There exists a strictly positive diagonal matrix $\mathbf{D}$ such that $\mathbf{D}\mathbf{A} \mathbf{D}$ is symmetric doubly stochastic.
\end{theorem} 
With the theorem above, define the map $\phi$ as follows:
\begin{align}
\phi: \mathds{S}\mathds{P}_n \times \overline{\mathds{R}}^{n} &\longrightarrow \overline{\mathcal{S}}_n \nonumber \\
 \left(\mathbf{A},d\right) &\longmapsto \text{diag}(d) \mathbf{A} \text{diag}(d).
\end{align}

Similar to the previous proof, $\overline{\mathcal{S}}_n$ is an open subset of the vector space $\mathcal{S}_n$ and $\overline{\mathds{R}}^{n}$ is a manifold. The dimension of the left hand side gives:
\begin{align}
\text{dim}(\mathds{S}\mathds{P}_n) + \text{dim}(\overline{\mathds{R}}^{n})  &= \cfrac{n(n-1)}{2} +n \nonumber \\
&= \cfrac{n(n+1)}{2}= \text{dim}(\mathcal{S}_n). 
\end{align}

Using similar techniques as the ones in \thref{th4}, we can conclude that $\phi$ is a diffeomorphism whose inverse is ensured by the DAD algorithm. Finally, one can note that the projection onto the set of symmetric double stochastic matrices leaves $\mathbf{X} + \xi_{\mathbf{X}}$ unchanged for $\xi_{\mathbf{X}}$ in the neighborhood of $\mathbf{0}_{\mathbf{X}}$.

\section{Riemannian Hessian Computation} \label{app4}

Recall that the Riemannian Hessian is related to the Riemannian connection and Riemnanian gradient through the following equation:
\begin{align}
\text{hess }f(\mathbf{X})[\xi_{\mathbf{X}}] = \nabla_{\xi_{\mathbf{X}}} \text{grad }f(\mathbf{X}), \ \forall \ \xi_\mathbf{X} \in \mathcal{T}_\mathbf{X}\mathcal{M}.
\end{align}
Furthermore, the connection $\nabla_{\eta_{\mathbf{X}}} \xi_{\mathbf{X}}$ on the submanifold is given by the Levi-Civita connection $\overline{\nabla}_{\eta_{\mathbf{X}}} \xi_{\mathbf{X}}$ on $\mathds{R}^{n \times n}$ by $\nabla_{\eta_{\mathbf{X}}} \xi_{\mathbf{X}}= \Pi_\mathbf{X}(\overline{\nabla}_{\eta_{\mathbf{X}}} \xi_{\mathbf{X}})$. Substituting in the expression of the Riemannian Hessian yields:
\begin{align}
&\text{hess }f(\mathbf{X})[\xi_{\mathbf{X}}] = \Pi_\mathbf{X}\left(\text{D}(\text{grad }f(\mathbf{X}))[\xi_{\mathbf{X}}]\right) \nonumber \\
&\qquad \qquad - \cfrac{1}{2}\Pi_\mathbf{X}\left(  (\text{grad }f(\mathbf{X}) \odot \xi_{\mathbf{X}}) \oslash \mathbf{X}\right) \nonumber \\
&= \Pi_\mathbf{X}\left(\text{D}(\text{grad }f(\mathbf{X}))[\xi_{\mathbf{X}}]\right)  \\
&\qquad \qquad - \cfrac{1}{2}\Pi_\mathbf{X}\left(  (\Pi_\mathbf{X}(\text{Grad }f(\mathbf{X}) \odot \mathbf{X}) \odot \xi_{\mathbf{X}}) \oslash \mathbf{X}\right)\nonumber
\end{align}

Apart for the term $\text{D}(\text{grad }f(\mathbf{X}))[\xi_{\mathbf{X}}]$, all the other terms in the above equation are available. Therefore, one only needs to derive the expression of $\text{D}(\text{grad }f(\mathbf{X})[\xi_{\mathbf{X}}]$ to obtain the mapping from the Euclidean gradient and Hessian to their Riemannian counterpart. For ease of notation, the section uses the short notation $\dot{f}[\xi]$ to denote the directional derivative $\text{D}(f)[\xi]$ (also denoted by $\xi f$ in the Riemannian geometry community).The computation of the directional derivative of the Riemannian gradient uses the result of the following proposition:
\begin{proposition}
Let $f$ and $g$ be two matrix functions, i.e., $f,g: \mathds{R}^{n \times n} \longrightarrow \mathds{R}^{n \times n}$. The directional derivative of the Hamadard product $f\odot g$ and the matrix product $fg$ are given by:
\begin{align}
\text{D}(f\odot g)[\xi] &=   \dot{f}[\xi] \odot g + f \odot \dot{g}[\xi] \\
\text{D}(fg)[\xi] &=   \dot{f}[\xi] g + f  \dot{g}[\xi]
\end{align}
\end{proposition}

\begin{proof}
The matrices identities and differentiation, including the above identities, are summarized in the following reference \cite{Petersen06thematrix}.
\end{proof}

The next subsections derive such directional derivative for the double stochastic and the symmetric multinomial manifold to derive the final expression of the Riemannian Hessian. In both subsections, let $\gamma$ denote $\text{Grad }f(\mathbf{X}) \odot \mathbf{X}$.
    
\subsection{Proof of \thref{th6}}

Recall that the projection on the set of doubly stochastic multinomial manifold is given by:
\begin{align}
\Pi_\mathbf{X}(\mathbf{Z}) &= \mathbf{Z} - (\alpha \mathbf{1}^T + \mathbf{1} \beta^T ) \odot \mathbf{X} \nonumber \\
\alpha &= (\mathbf{I} - \mathbf{X}\mathbf{X}^T)^{\dagger}(\mathbf{Z} - \mathbf{X}\mathbf{Z}^T)\mathbf{1} \nonumber \\
\beta &= \mathbf{Z}^T\mathbf{1} - \mathbf{X}^{T}\alpha
\end{align}
Therefore, the directional derivative can be expressed as:
\begin{align}
&\text{D}(\text{grad }f(\mathbf{X}))[\xi_{\mathbf{X}}] = \text{D}(\Pi_\mathbf{X}(\gamma)[\xi_{\mathbf{X}}] \nonumber \\
&= \text{D}(\gamma-(\alpha \mathbf{1}^T + \mathbf{1} \beta^T ) \odot \mathbf{X})[\xi_{\mathbf{X}}] \nonumber \\
&= \text{D}(\gamma)[\xi_{\mathbf{X}}]-\text{D}((\alpha \mathbf{1}^T + \mathbf{1} \beta^T ) \odot \mathbf{X})[\xi_{\mathbf{X}}] \nonumber \\
&= \dot{\gamma}[\xi_{\mathbf{X}}] -(\dot{\alpha}[\xi_{\mathbf{X}}] \mathbf{1}^T + \mathbf{1} \dot{\beta}^T[\xi_{\mathbf{X}}] ) \odot \mathbf{X} \nonumber \\
&\qquad \qquad- (\alpha \mathbf{1}^T + \mathbf{1} \beta^T ) \odot \xi_{\mathbf{X}}
\end{align}
with
\begin{itemize}
\item $\dot{\gamma}[\xi_{\mathbf{X}}] = \text{D}(\gamma)[\xi_{\mathbf{X}}]$ can be expressed as:
\begin{align}
\dot{\gamma}[\xi_{\mathbf{X}}] &= \text{D}(\text{Grad }f(\mathbf{X}))[\xi_{\mathbf{X}}] \odot \mathbf{X} + \text{Grad }f(\mathbf{X}) \odot \xi_{\mathbf{X}}  \nonumber \\
&= \text{Hess }f(\mathbf{X})[\xi_{\mathbf{X}}]\odot \mathbf{X} + \text{Grad }f(\mathbf{X}) \odot \xi_{\mathbf{X}} \nonumber
\end{align}
\item $\dot{\alpha}[\xi_{\mathbf{X}}] =  \text{D}(\alpha)[\xi_{\mathbf{X}}]$ can be computed as follows:
\begin{align}
\dot{\alpha}[\xi_{\mathbf{X}}] &=  \text{D}((\mathbf{I} - \mathbf{X}\mathbf{X}^T)^{\dagger}(\gamma - \mathbf{X}\gamma^T)\mathbf{1})[\xi_{\mathbf{X}}]  \\
&=  \text{D}((\mathbf{I} - \mathbf{X}\mathbf{X}^T)^{\dagger})[\xi_{\mathbf{X}}](\gamma - \mathbf{X}\gamma^T)\mathbf{1} \nonumber \\
& \qquad + (\mathbf{I} - \mathbf{X}\mathbf{X}^T)^{\dagger}(\dot{\gamma}[\xi_{\mathbf{X}}]-\xi_{\mathbf{X}}\gamma-\mathbf{X}\dot{\gamma}^T[\xi_{\mathbf{X}}])\mathbf{1}\nonumber,
\end{align}
with the term $\text{D}((\mathbf{I} - \mathbf{X}\mathbf{X}^T)^{\dagger})[\xi_{\mathbf{X}}]$ being derived below.
\item $\dot{\beta}[\xi_{\mathbf{X}}] =  \text{D}(\beta)[\xi_{\mathbf{X}}]$ can be computed as follows:
\begin{align}
\dot{\beta}[\xi_{\mathbf{X}}] &= \text{D}(\gamma^T\mathbf{1} - \mathbf{X}^{T}\alpha)[\xi_{\mathbf{X}}] \nonumber \\
&= \dot{\gamma}^T[\xi_{\mathbf{X}}]\mathbf{1} - \xi_{\mathbf{X}}^T\alpha- \mathbf{X}^{T}\dot{\alpha}[\xi_{\mathbf{X}}]
\end{align}
\end{itemize}

In order to compute $\text{D}((\mathbf{I} - \mathbf{X}\mathbf{X}^T)^{\dagger})[\xi_{\mathbf{X}}]$, first introduce the following lemma:
\begin{lemma}
Let $\mathbf{A}$ be an $n \times n$ matrix with a left pseudo inverse $\mathbf{A}^\dagger$. The left pseudo inverse of $(\mathbf{A}+\mathbf{B}\mathbf{C})$ is given by:
\begin{align}
(\mathbf{A}+\mathbf{B}\mathbf{C})^\dagger = \mathbf{A}^\dagger - \mathbf{A}^\dagger\mathbf{B}(\mathbf{I}+\mathbf{C}\mathbf{A}^\dagger\mathbf{B})^\dagger \mathbf{C} \mathbf{A}^\dagger
\end{align}
\end{lemma}

\begin{proof}
The above identity is similar to the Kailath variant of Sherman-Morrison-Woodbury formula \cite{57151854} for an invertible matrix $\mathbf{A}$. The proof is given by a simple left multiplication as follows:
\begin{align}
&(\mathbf{A}+\mathbf{B}\mathbf{C})^\dagger(\mathbf{A}+\mathbf{B}\mathbf{C}) = \mathbf{A}^\dagger\mathbf{A} - \mathbf{A}^\dagger\mathbf{B}(\mathbf{I}+\mathbf{C}\mathbf{A}^\dagger\mathbf{B})^\dagger \mathbf{C} \mathbf{A}^\dagger\mathbf{A} \nonumber \\
& \qquad \qquad + \mathbf{A}^\dagger \mathbf{B}\mathbf{C} - \mathbf{A}^\dagger\mathbf{B}(\mathbf{I}+\mathbf{C}\mathbf{A}^\dagger\mathbf{B})^\dagger \mathbf{C} \mathbf{A}^\dagger\mathbf{B}\mathbf{C}  \\
& = \mathbf{I} - \mathbf{A}^\dagger \mathbf{B}((\mathbf{I}+\mathbf{C}\mathbf{A}^\dagger\mathbf{B})^\dagger(\mathbf{I}+\mathbf{C}\mathbf{A}^\dagger\mathbf{B})\mathbf{C}-\mathbf{C} ) = \mathbf{I}\nonumber
\end{align}
\end{proof}

Using the identity above, the pseudo inverse of the perturbed $(\mathbf{I} - \mathbf{X}\mathbf{X}^T)^{\dagger}$ along $\xi_{\mathbf{X}}$ is given by:
\begin{align}
&(\mathbf{I} - (\mathbf{X}+t\xi_{\mathbf{X}})(\mathbf{X}+t\xi_{\mathbf{X}})^T)^{\dagger} = \mathbf{A}[\xi_{\mathbf{X}}]^\dagger \nonumber \\
&\qquad \qquad= \mathbf{A}^\dagger + t\mathbf{A}^\dagger(\mathbf{I} - t\mathbf{C}\mathbf{A}^\dagger)^\dagger \mathbf{C} \mathbf{A}^\dagger
\end{align}
wherein $\mathbf{A} = \mathbf{I} - \mathbf{X}\mathbf{X}^T$, $\mathbf{B} = -t$, and $\mathbf{C} = \mathbf{X}\xi_{\mathbf{X}}^T + \xi_{\mathbf{X}} \mathbf{X}^T + t\xi_{\mathbf{X}}\xi_{\mathbf{X}}^T$ in the above inversion lemma. Therefore, the directional derivative can obtained by:
\begin{align}
\text{D}((\mathbf{I} - \mathbf{X}\mathbf{X}^T)^{\dagger})&[\xi_{\mathbf{X}}] = \lim_{t \rightarrow 0} \cfrac{\mathbf{A}[\xi_{\mathbf{X}}]^\dagger-\mathbf{A}^\dagger}{t} \nonumber \\
&= \lim_{t \rightarrow 0} \cfrac{t\mathbf{A}^\dagger(\mathbf{I}-t\mathbf{C}\mathbf{A}^\dagger)^\dagger \mathbf{C} \mathbf{A}^\dagger}{t} \nonumber \\
&= \lim_{t \rightarrow 0} \mathbf{A}^\dagger(\mathbf{I}-t\mathbf{C}\mathbf{A}^\dagger)^\dagger \mathbf{C} \mathbf{A}^\dagger \nonumber \\
&=  \mathbf{A}^\dagger(\lim_{t \rightarrow 0}\mathbf{C}) \mathbf{A}^\dagger  \\
&= (\mathbf{I} - \mathbf{X}\mathbf{X}^T)^\dagger (\mathbf{X}\xi_{\mathbf{X}}^T + \xi_{\mathbf{X}} \mathbf{X}^T)(\mathbf{I} - \mathbf{X}\mathbf{X}^T)^\dagger \nonumber
\end{align}

\subsection{Proof of \corref{cor3}}

The proof of this corollary follows similar steps as the one of \thref{th6}. Recall that the projection on the symmetric multinomial manifold is given by:
\begin{align}
\Pi_\mathbf{X}(\mathbf{Z}) &= \mathbf{Z} - (\alpha \mathbf{1}^T + \mathbf{1} \alpha^T ) \odot \mathbf{X} \nonumber \\
\alpha &= (\mathbf{I} + \mathbf{X})^{-1}\mathbf{Z}\mathbf{1}.
\end{align}
Therefore, using a technique similar to \thref{th6}, the directional derivative can be expressed as:
\begin{align}
&\text{D}(\text{grad }f(\mathbf{X}))[\xi_{\mathbf{X}}] = \dot{\gamma}[\xi_{\mathbf{X}}] -(\dot{\alpha}[\xi_{\mathbf{X}}] \mathbf{1}^T + \mathbf{1} \dot{\alpha}^T[\xi_{\mathbf{X}}] ) \odot \mathbf{X} \nonumber \\
&\qquad \qquad- (\alpha \mathbf{1}^T + \mathbf{1} \alpha^T ) \odot \xi_{\mathbf{X}}
\end{align}
with $\dot{\gamma}[\xi_{\mathbf{X}}] = \text{Hess }f(\mathbf{X})[\xi_{\mathbf{X}}]\odot \mathbf{X} + \text{Grad }f(\mathbf{X}) \odot \xi_{\mathbf{X}}$. The computation of the directional derivative of $\alpha$ requires differentiating $(\mathbf{I} + \mathbf{X})^{-1} = \mathbf{A}^{-1}$. Using the Kailath variant of Sherman-Morrison-Woodbury formula \cite{57151854}, the inverse is given by:
\begin{align}
(\mathbf{I} + \mathbf{X} + t\xi_{\mathbf{X}})^{-1} = \mathbf{A}^{-1}-\mathbf{A}^{-1}t(\mathbf{I}+t\xi_{\mathbf{X}}\mathbf{A}^{-1})^{-1}\xi_{\mathbf{X}}\mathbf{A}^{-1}.
\end{align}
Therefore, the directional derivative can be expressed as:
\begin{align}
\text{D}((\mathbf{I} + \mathbf{X})^{-1})[\xi_{\mathbf{X}}] &= \lim_{t \rightarrow 0} \cfrac{\mathbf{A}[\xi_{\mathbf{X}}]^{-1}-\mathbf{A}^{-1}}{t} \nonumber \\
&= \lim_{t \rightarrow 0} \cfrac{-\mathbf{A}^{-1}t(\mathbf{I}+t\xi_{\mathbf{X}}\mathbf{A}^{-1})^{-1}\xi_{\mathbf{X}}\mathbf{A}^{-1}}{t}\nonumber \\
&= -(\mathbf{I} + \mathbf{X})^{-1}\xi_{\mathbf{X}}(\mathbf{I} + \mathbf{X})^{-1}
\end{align}
Hence, we obtain:
\begin{align}
\dot{\alpha}[\xi_{\mathbf{X}}] = \left((\mathbf{I} + \mathbf{X})^{-1}\dot{\gamma}[\xi_{\mathbf{X}}] -(\mathbf{I} + \mathbf{X})^{-1}\xi_{\mathbf{X}}(\mathbf{I} + \mathbf{X})^{-1}\gamma\right)\mathbf{1}. \nonumber
\end{align}

\bibliographystyle{IEEEtran}
\bibliography{citations}

\end{document}